\documentclass[11pt,reqno]{amsart}
\usepackage{amscd}
\usepackage{amsfonts}
\usepackage{amsmath,amsthm,hyperref}
\usepackage[all]{xy}
\usepackage{amsmath,amsthm,hyperref}
\usepackage{amsmath,amssymb,amsthm,latexsym}
\usepackage{color}
\usepackage{amsmath}
\usepackage{bm}
\usepackage{dsfont}
\usepackage{changes}
\usepackage{enumerate}
\numberwithin{equation}{section}
\usepackage{ulem}
\newtheorem{theorem}{Theorem}[section]
\newtheorem{proposition}[theorem]{Proposition}

\newtheorem{corollary}[theorem]{Corollary}
\newtheorem{lemma}[theorem]{Lemma}
\newtheorem{remark}[theorem]{Remark}
\newtheorem{conjecture}[theorem]{Conjecture}

\usepackage{caption}
\usepackage{float}
\usepackage{graphicx}
\usepackage{CJK}
\usepackage{amsmath,amsfonts,mathrsfs,amssymb}
\usepackage{indentfirst}
\usepackage{caption}

\usepackage{multicol}
\usepackage{longtable}
\usepackage[numbers,sort&compress]{natbib}
\usepackage{xcolor}
\usepackage{indentfirst}
\usepackage{amscd}
\usepackage{amsfonts}
\usepackage{amsmath,amsthm,hyperref}
\usepackage[all]{xy}
\usepackage{amsmath,amsthm,hyperref}
\usepackage{amsmath,amssymb,amsthm,latexsym}
\usepackage{amscd}
\theoremstyle{remark}

\newcommand{\R}{\mathbb{R}}

\newcommand{\SB}{\mathbb{S}}

\newcommand{\CC}{\mathcal{T}^n}
\newcommand{\Z}{\mathbb{Z}}

\newcommand{\dd}{\mathrm d} 

\newcommand{\blue}{\textcolor{blue}}

\newcommand{\HT}{\CJKfamily{hei}}

\newcommand{\W}{\mathcal{W}}

\makeatletter

\newcommand*{\innerproduct}[2]{%
  \if@display
    \left\langle #1,#2\right\rangle
  \else
    \langle #1,#2\rangle
  \fi
}

\renewcommand{\@settitle}{%
  \begin{center}
    \baselineskip 18\p@\relax
    \Large\bfseries
    \@title
  \end{center}%
}

\makeatother
    
\begin{document}

\title{
On the Willmore energy of flat $n$-tori in $\mathbb{R}^N$ and Chen's conjecture for $n$-tori
}
\author{
Ruijie Ni, Peng Wang, Zhenxiao Xie}
\address{School of Mathematics and Statistics, FJKLAMA, Key Laboratory of Analytical Mathematics and Applications (Ministry of Education), Fujian Normal University, Fuzhou, China} \email{niruijie721@sina.com}

\address{School of Mathematics and Statistics, FJKLAMA, Key Laboratory of Analytical Mathematics and Applications (Ministry of Education), Fujian Normal University, Fuzhou, China}
\email{pengwang@fjnu.edu.cn, netwangpeng@163.com}

\address{School of Mathematical Sciences, Beihang University, Beijing 102206, China} 
\email{xiezhenxiao@buaa.edu.cn}
\maketitle

  \begin{abstract}
 This paper establishes the sharp lower bound $(4n\pi^2)^{n/2}$ for the Willmore energy $\mathcal{W}$ of flat $n$-tori in the Euclidean space. Up to M\"obius transformations, the Clifford $n$-torus $\SB^1\bigl(\sqrt{1/n}\,\bigr) \times \cdots \times \SB^1\bigl(\sqrt{1/n}\,\bigr) \subset \SB^{2n-1} \subset \mathbb{R}^{2n}$ is shown to be the unique minimizer attaining this bound. This also confirms Chen's conjecture for flat $n$-tori. However, when $n 
 \geq3 $, we show that Chen's conjecture fails on the total mean curvature of general immersed $n$-tori: certain M\"{o}bius transformations of the Clifford $n$-torus strictly decrease the  total mean curvature. 
  \end{abstract}

 
  \vspace{2mm}

   \vspace{2mm}

\section{Introduction}
Conformal geometry is an important branch of differential geometry. The Willmore energy is a basic conformally invariant curvature functional for submanifolds. Its lower-bound estimates, as a natural problem in global differential geometry, have received much attention.

For a closed surface $M^2$ in the Euclidean space $\mathbb R^N$, the Willmore energy is defined by
\[\widetilde{\mathcal W}(M^2)=\int_{M^2} |H|^2\,\mathrm{d}V,\]
where $H$ denotes the mean curvature vector of $M$ and $\mathrm{d}V$ is the
area form induced by the immersion. In 1965, Willmore proposed the celebrated
conjecture that every  torus in $\mathbb R^3$ satisfies
$\widetilde{\mathcal W}(T^2)\geq 2\pi^2,$
with equality if and only if the torus is conformally congruent to the
Clifford torus. The Willmore conjecture has inspired substantial work; see,
for instance, \cite{W-1965,Li-Yau,M-R,B-1984,K-1989,B-2015,Chen1973,Mo-Ng-2014}. It was finally
proved by Marques and Neves in \cite{M-N1,M-N2}, using methods from geometric
measure theory and minimal surface theory. 
Nevertheless, sharp lower-bound problems for the Willmore energy remain open in several settings, including higher codimension and higher genus surfaces.

Let $f:M^n\to\mathbb R^N$ be an $n$-dimensional closed submanifold with $n>2$. The  Willmore functional of $f$ is defined by \cite{Chen1974,Li-Yau,B-1988,PeditWillmore1988,Wang1998}
\begin{equation}\label{eq-GW}
\mathcal W(M^n)\triangleq
\left(\frac{n}{n-1}\right)^{\frac n2}
\int_{M^n}
\left(S-n|H|^2\right)^{\frac n2}\mathrm{d}V,
\end{equation}
where $S$ denotes the squared norm of the second fundamental form of $f$ in $\mathbb R^N$, and $H$ denotes its mean curvature vector. This functional is invariant under conformal transformations of $\mathbb R^N$; see \cite{Chen1974,PeditWillmore1988,Wang1998}. More generally, Mondino and Nguyen investigated which curvature integrals of submanifolds remain invariant under conformal changes of the ambient metric. In particular, they showed that, for codimension-one surfaces, the Willmore energy is the unique global conformal invariant, up to the addition of a Gauss--Bonnet term \cite{Mo-Ng-2018}, i.e.,
$$\mathcal W(M^2)
=4\widetilde{\mathcal W}(M^2)-8\pi\chi(M^2).
$$
In the remainder of this paper, we will always use $\W(M^n)$ as the Willmore functional/energy.

For general submanifolds, Pinkall provided lower-bound estimates for the Willmore functional  via  Willmore-type inequalities for submanifolds \cite{Pinkall1986}.
The critical submanifolds of the Willmore functional are called Willmore submanifolds. 
Basic examples and the variational theory including stability and so on, have been investigated in \cite{Li2002,GuoLiWang2001,Li2001,QianTangYan2013,TangYan2012,R-1993,Wang1998}.
These works have developed important examples and tools for the study of  conformal geometry of submanifolds. Nevertheless, sharp lower bounds for the Willmore energy remain largely unknown, particularly in higher codimensions.


In this paper, we focus on the lower bound of the Willmore energy for a special class of submanifolds  in $\mathbb R^N$, namely flat $n$-dimensional tori ($n$-tori for short). By the Gauss equation, for such submanifolds, the Willmore energy reduces to the total mean curvature $\int_{T^n} |H|^n\,\mathrm{d}V$(up to the multiplicative constant $n^n$).
Concerning this functional,  Chen [\citealp[Theorem~3]{Chen1971}; \citealp[Theorem~3]{Chen1971TotalI}] proved that for every closed immersion $f:M^n\to\mathbb{R}^N$, one has $$\int_{M^n}|H|^n\,\mathrm{d}V\geq\omega_n,$$
with equality holding if and only if $f(M^n)$ is a round $n$-sphere in $\R^{n+1}\subset \mathbb{R}^N$ when $n\ge2$. Here $\omega_n$ denotes the volume of the unit $n$-sphere. See \cite[Theorem~2.1]{Brendle2026} for another proof by Brendle based on the Alexandrov-Bakelman-Pucci technique. When $n=2$, this goes back to the famous Willmore inequality \cite{W-1965}.
Inspired by Marques and Neves' resolution of the  Willmore conjecture in $\SB^3$ \cite{M-N1,M-N2} and also  Chen's works on the total mean curvature \cite{Chen1972,Chen1974,Chen1979}, Chen \cite{Chen2015} proposed the  following conjecture in 2015. 
 \begin{conjecture}(Conjecture 5.5 of \cite[Page 132]{Chen2015})\label{conj-chen}
For any immersion of the $n$-torus $T^n$ into Euclidean space, the following inequality holds,  
\begin{equation}
    \label{eq-chen}\int_{T^n} |H|^n\,\mathrm{d}V
    \geq
    \left(\frac{4\pi^2}{n}\right)^{n/2}.
\end{equation}
\end{conjecture}
\noindent 
Note that this bound is exactly the volume of the Clifford $n$-torus 
\[\CC=\SB^1\left(\sqrt{\frac{1}{n}}\right)\times\cdots\SB^1\left(\sqrt{\frac{1}{n}}\right)\subset \SB^{2n-1}.\] When $n=2$, this was obtained for flat $2$-tori by Chen \cite{Chen1981} following the method of Chern-Lashof \cite{ChernLashof1958}, and also by Li-Yau using Fourier expansion and an algebraic inequality \cite{Li-Yau}. 

Applying the Korkine–Zolotarev reduction from lattice theory, we can generalize Li–Yau's algebraic inequality 
to the cases $3\leq n \leq 5$, and then establish a sharp lower bound for the Willmore energy of flat $n$-tori in these dimensions, which also confirms Chen's conjecture in the flat case. 
However, for $n\geq 6$, the Korkine–Zolotarev reduction no longer yields the required algebraic inequality (see Remark~\ref{rk-4.6}). Nevertheless, we show that a global estimate for the Fourier expansion of the immersion can be obtained via finite truncation, in contrast to Li-Yau's termwise estimates. This approach allows us to establish the sharp lower bound of the Willmore energy for flat $n$-tori in every dimension, as well as a proof of Chen's conjecture in the flat case.



However, by applying suitable M\"{o}bius transformations to the Clifford $n$-torus, we show that Chen's conjecture does not hold for  non-flat $n$-tori when $n\geq 3$. More precisely, there exists a one-parameter family of M\"{o}bius transformations of $\mathbb S^{2n-1}$ such that the total mean curvature strictly decreases along the flow at least for $t$ sufficiently small.  
\medskip

\noindent{\bf Plan of the paper.}
In Section \ref{section-pre}, we recall the basic definitions of flat tori and list several formulas needed in the proofs. In Section \ref{section-tn}, we study arbitrary isometric immersions of flat tori into $\mathbb R^N$. We prove that the Clifford torus attains the sharp lower bound for the Willmore energy. In Section \ref{section-KZ}, we apply the Korkine--Zolotarev reduction to $T^n$ ($3\le n\le 5$) and prove a sharp Willmore inequality, with the rectangular torus as the equality case. In Section \ref{Section-chen-conj}, we  establish Chen's conjecture for the flat case and show that the conjecture does not hold for general immersions when $n\geq3$. 

\section{Preliminaries}\label{section-pre}
We collect some basic facts on flat $n$-tori that will be used in the sequel: their moduli space and Fourier series. In addition, we shall also recall the Cauchy--Binet theorem. Throughout the paper, vectors in $\mathbb{R}^n$ are regarded as column vectors.
\subsection{Flat tori and their moduli space}~

Let $\{v_1,\ldots,v_n\}$ be a basis of $\mathbb R^n$. Set
$$\Lambda=\{m_1v_1+\cdots+m_nv_n:\ m_i\in\mathbb Z\}=V\mathbb Z^n,$$
 with $V=(v_1,\ldots,v_n)\in GL(n,\R)$.
The quotient space $T_V^n=\mathbb{R}^n/V\mathbb{Z}^n=\mathbb{R}^n/\Lambda$, endowed with the Euclidean metric $g_0$ from $\R^n$, is a flat $n$-torus. 

The choice of the basis matrix $V$ is not unique. If $U\in GL(n,\mathbb Z)$, then $VU$ generates the same lattice as $V$, while multiplication by $O\in O(n)$ produces an isometric lattice. Therefore, the moduli space of the isometry class of $T^n_V$ can be described as
$$\mathcal{M}_n^{\mathrm{flat}}=O(n)\backslash GL(n,\mathbb{R})/GL(n,\mathbb{Z}).$$
On the other hand, replacing $V$ by $e^{\rho_0} V$, $\rho_0>0$, only multiplies the flat metric by the constant factor $e^{2\rho_0}$ and hence leaves the conformal class unchanged. 

\subsection{Fourier series on flat tori}~

For the torus $T_V^n=\mathbb{R}^n/V\mathbb{Z}^n=\mathbb{R}^n/\Lambda_V$, its dual lattice is 
$$\Lambda_V^*=\{\xi\in\mathbb{R}^n:\langle \xi,\lambda\rangle\in\mathbb{Z} \text{ for all } \lambda\in\Lambda_V\},$$ 
where $\langle \xi,\lambda\rangle=\xi^T\lambda$ denotes the natural pairing.
The spectrum of $T_V^n$ is given by
$$\operatorname{Spec}(T_V^n)=\{4\pi^2|\xi|^2:\xi\in\Lambda_V^*\}.$$
Moreover, for each $\xi\in\Lambda_V^*$, the function $e^{2\pi i\langle \xi,x\rangle}$ is an eigenfunction corresponding to the eigenvalue $4\pi^2|\xi|^2$, where $x=(x_1,\ldots,x_n)$ denotes the standard coordinates in $\mathbb R^n$ and  $g_0=\mathrm{d}x_1^2+\cdots+\mathrm{d}x_n^2$. 

The simplest flat $n$-torus is given by $V = I_n$, where $I_n$ denotes the identity matrix. In this case, $T_{\mathrm{I}_n}^n = \mathbb{R}^n / \mathbb{Z}^n$ is the product of $n$ copies of the unit circle, and it admits a minimal homothetic embedding in $\mathbb{S}^{2n-1}$ via  
\begin{equation}
    \label{eq-CT}
    \frac{1}{\sqrt{n}}
    \bigl(
    \cos(2\pi x_1), \sin(2\pi x_1),
    \ldots,
    \cos(2\pi x_n), \sin(2\pi x_n)
    \bigr): T_{\mathrm{I}_n}^n\rightarrow \mathbb{S}^{2n-1}.
\end{equation}
This embedding will be referred to as the Clifford $n$-torus $\CC$ in $\mathbb{S}^{2n-1}$. 

For a general flat $n$-torus, let $f:T_V^n\to\mathbb R^N$ be an isometric immersion. Consider its Fourier expansion 
\begin{equation}\label{eq-FS}
   f(x)=\sum_{\xi\in\Lambda^*_V}f_\xi e^{2\pi i\langle \xi,x\rangle}, 
\end{equation}
where $f_\xi\in\mathbb C^N$ are the Fourier coefficients and satisfying $f_{-\xi}=\overline{f_\xi}$ since $f$ is real-valued.

Associated with the Fourier expansion above, we denote by
\[
\mathcal A(f):=\{\xi\in\Lambda_V^*:f_\xi\neq0\}
\]
the set of all frequencies appearing in the immersion $f$, and define its
nonzero part by
\[
\mathcal A^*(f):=\mathcal  A(f)\setminus\{0\}.
\]

The following two lemmas are used in proofs of Theorem \ref{thm-forall-n}, Corollary \ref{thm-mini-iso} and Theorem \ref{thm-iso}.
\begin{lemma}\label{lemma-iso condition}
   If $f:(T^n_V,e^{2\rho_0}g_0)\longrightarrow \mathbb{R}^N$ is an isometric immersion (here $\rho_0=const$), then 
    $$\mathrm{span}_{\mathbb{R}}\mathcal{A}(f)=\mathbb{R}^n.$$
    \end{lemma}
    \begin{proof}
The Fourier expansion of $f$ is given by (\ref{eq-FS}). Since $f$ is an isometric immersion, we get
\begin{equation}\label{eq-int-iso}
    \begin{split}
       e^{2\rho_0}\delta_{ij}&= \frac{1}{\mathcal{V}(f)}\int_{T^n_V}\frac{\partial f}{\partial x_i}\cdot\frac{\partial f}{\partial x_j} \mathrm{d}V=\frac{-1}{\mathcal{V}(f)}\int_{T^n_V}4\pi^2\sum_{\xi,\eta\in\Lambda^*_V}\xi_i\eta_jf_\xi f_\eta e^{2\pi i \langle\xi+\eta,x\rangle}\mathrm{d}V\\
        &=4\pi^2\sum_{\xi\in\Lambda^*_V}|f_\xi|^2\xi_i\xi_j=4\pi^2\sum_{\xi\in\mathcal{A}(f)}|f_\xi|^2\xi_i\xi_j .
    \end{split}
\end{equation}
It follows that for any $X\in\R^n$
\begin{equation}\label{eq-constraction}
   4\pi^2\sum_{\xi\in\mathcal{A}(f)}|f_\xi|^2\xi\xi^T X=e^{2\rho_0}I_nX=e^{2\rho_0}X,
\end{equation}
where $I_n$ denotes the identity matrix. We now argue by contradiction. Assume that for the isometric immersion $f$, the space $\operatorname{span}_{\mathbb{R}}\mathcal{A}(f)$ is a proper linear subspace of $\mathbb{R}^n$. Then there exists a nonzero vector $X\in\{\operatorname{span}_{\mathbb{R}}\mathcal{A}(f)\}^\perp\setminus\{\vec{0}\}$, that is, $    \langle\xi,X\rangle=\xi^TX=0$ for all $\xi\in\mathcal{A}(f)$. But this contradicts \eqref{eq-constraction}.
    \end{proof}
 
\begin{lemma}\label{lem-Cliff}
Let $f: (T_{\mathrm{I}_n}^n, g_0)\rightarrow \mathbb{R}^N$ be an isometric immersion with  the length of the mean curvature vector being constant. If  
$$\mathcal{A}(f)=\{\pm e_1, \pm e_2, \cdots, \pm e_n\},$$  
then up to a dilation, $f$ is congruent to the Clifford $n$-torus $\CC$ in $\mathbb{S}^{2n-1}$. 
\end{lemma}
\begin{proof}
Observe that all elements in $\mathcal{A}(f)$ are the shortest vectors of the lattice $\Lambda_I^*$, which implies that $f$ is an isometric immersion into $\mathbb{R}^{N}$ via the first eigenfunctions. Moreover, from  
$$-4\pi^2 f = \Delta f = n H,$$  
it follows that $|f|$ is a constant. 
Therefore, after a dilation,  we can regard $f$ as a minimal isometric immersion into $\mathbb{S}^{N-1}$. It is obvious that $\mathcal{A}(f)=\{\pm e_1,\cdots,\pm e_n\}$ satisfies the unimodular condition of Proposition 2.7 in \cite{L-W-X2024}. Consequently, $f$ is homogeneous, and up to an orthogonal transformation of $\mathbb{R}^N$, we can assume that 
$$f_{e_j}=a_j(E_{2j-1}-iE_{2j}),\quad 
1\leq j\leq n,$$
 where $E_1,\dots, E_N$ are the standard orthonormal basis of $\R^N$. Then it follows from the isometric condition 
$$\frac{I_n}{16\pi^2}=\sum_{j=1}^n a_j^2\, e_j e_j^t$$
that 
$a_j=\frac{1}{4\pi}$ for all $1\leq j\leq n$. 
Hence, up to a dilation, $f$ is congruent to the standard Clifford $n$-torus $\CC$ in $\mathbb{S}^{2n-1}$. 
\end{proof}
 
\subsection{The Cauchy--Binet Formula}~

We now recall the Cauchy--Binet formula for matrices \cite{B-W1989} and introduce the notation that will be used below.

Let $B=(b_{ij})$ be an $n\times m$ matrix and set $1\leq k\leq\min\{m,n\}$. Denote $[n]=\{1,\ldots,n\}$ and $[m]=\{1,\ldots,m\}$. For two index sets $W=\{i_1<i_2<\cdots<i_k\}\subset [n]$ and $T=\{j_1<j_2<\cdots<j_k\}\subset [m]$,
we denote by $B_{W,T}$ the submatrix of $B$ obtained by taking the rows indexed by
$W$ and the columns indexed by $T$, that is
$$ B_{W,T} =\begin{pmatrix}
b_{i_1j_1} & b_{i_1j_2} & \cdots & b_{i_1j_k} \\
b_{i_2j_1} & b_{i_2j_2} & \cdots & b_{i_2j_k} \\
\vdots & \vdots & \ddots & \vdots \\
b_{i_kj_1} & b_{i_kj_2} & \cdots & b_{i_kj_k}
\end{pmatrix}.
$$

The determinant $\det( B_{W,T})$ is called a $k$-minor of $B$. 
We shall use the following form of the Cauchy--Binet formula.

\begin{theorem}\label{thm-CB}\cite{B-W1989}
Let $B,C$ be two $n\times m$ matrices. If $m<n$, $\det(BC^T)=0.$
If $m\geq n$, 
\begin{equation}\label{eq:CB}
\det(BC^T)
=
\sum_{\substack{T\subset [m]\\ |T|=n}}
\det(B_{[n],T})\det(C_{[n],T}).    
\end{equation}

\end{theorem}

When $B=C=(b_1, b_2, \cdots, b_m)$ with $b_j\in\mathbb R^n$, $1\leq j\leq m$, we get $BB^T=\sum_{\alpha=1}^m b_\alpha b_\alpha^T$ and
$$
\det(BB^T)
=
\sum_{\substack{T\subset [m]\\ |T|=n}}
\det(B_{[n],T})^2=
\sum_{\substack{T=\{\alpha_1<\cdots<\alpha_n\}\subset [m]}}
\det(b_{\alpha_1},\ldots,b_{\alpha_n})^2,
$$
which means that $\det(BB^T)$ is equal to the sum of the squares of the $n$-dimensional volumes
spanned by all possible choices of $n$ column vectors of $B$.

We also need the following useful lemma.
\begin{lemma}\label{lemma-matrix}\cite{H-J2012}
Let $B_1$ be positive definite and $B_2$ be positive semidefinite. If $B_1- B_2$ is a positive semidefinite matrix, then $$ \det B_1\geq \det B_2.$$
\end{lemma}

\section{On the Willmore problem for  flat $T^n$ in $\R^{N}$}\label{section-tn}

 For a flat $n$-torus in the Euclidean space, the Willmore functional \eqref{eq-GW} becomes
\begin{equation}\label{eq-W-iso}
    \mathcal{W}(f) = n^n \int_{T_V^n} |H|^n \, \mathrm{d}V.
\end{equation}
Thus, in the flat case, estimating the conformal Willmore functional reduces to estimating the total mean curvature. In this section, we adopt a Fourier series approach and estimate the contribution of the dual lattice in a global manner. This yields a sharp lower bound for the Willmore energy of flat $n$-tori, which is achieved by the Clifford $n$-torus $\CC$ in $\mathbb{S}^{2n-1}$. The precise statement is as follows.

\begin{theorem}\label{thm-forall-n}
 Let $f:(T^n_V,g_0)\to \mathbb{R}^N$ be an immersion such that there exists a  M\"{o}bius transformation $\varphi$ of $\mathbb{R}^N$ satisfying  $(\phi\circ f)^*g_{\mathbb R^N} = g_0$.
Then \[\W(f)\ge (4n\pi^2)^{n/2},\] with equality if and only if $f$ is M\"{o}bius congruent to the Clifford torus $\CC\subset \mathbb{S}^{2n-1}\subset \mathbb{R}^N$.
\end{theorem}

 
It follows from H\"older's inequality that
\begin{equation}\label{eq-Holder}
   \int_{T^n_V}|H|^n\,\mathrm{d}V \geq \bigl(\operatorname{Vol}(T^n_V)\bigr)^{1-\frac{n}{2}} \biggl(\int_{T^n_V}|H|^2\,\mathrm{d}V\biggr)^{\frac{n}{2}}.
\end{equation}
For the integral on the right-hand side, using $\Delta f= n H$ and the Fourier expansion of $f$ given in \eqref{eq-FS}, one can verify directly that
\begin{equation}\label{eq-H}
\begin{split}
   \int_{T^n_V}|H|^2\,\mathrm{d}V = \frac{1}{n^2}\int_{T^n_V}|\Delta f|^2\,\mathrm{d}V = \frac{16}{n^2}\pi^4\,\operatorname{Vol}(T^n_V)\sum_{\xi\in\Lambda^*_V}|f_{\xi}|^2|\xi|^4.
\end{split}
\end{equation}
Then the lower bound of $\mathcal{W}(f)$ follows from the following lemma.

\begin{lemma}\label{lemma-H-estimate}
    For any isometric immersion $f \colon (T^n_V, g_0) \longrightarrow \mathbb{R}^N$, we have 
\begin{equation}\label{ineq-HV-estimate}
\sum_{\xi \in \Lambda^*_V} 4\pi^2 |f_\xi|^2 |\xi|^4 \geq n \operatorname{Vol}(T^n_V)^{-2/n}.
\end{equation}
Equality holds if and only if
\[
\Lambda^*_V = \operatorname{Span}_\mathbb{Z}\{\eta_1, \eta_2, \dots, \eta_n\},
\]
where $\eta_1, \eta_2, \dots, \eta_n$ are mutually orthogonal vectors of the same length in $\mathbb{R}^n$, with $\mathcal{A}^*(f) = \{\pm\eta_1, \dots, \pm\eta_n\}$.  
\end{lemma}
\begin{proof}

Set $$C\triangleq \sum_{\xi\in \Lambda^*_V}4\pi^2 |f_\xi|^2 |\xi|^2\,\xi\xi^T,$$ whose trace is given by 
$$\operatorname{tr}C=\sum_{\xi\in \Lambda^*_V}4\pi^2 |f_\xi|^2 |\xi|^4=\sum_{\xi\in \mathcal{A}^*(f)}4\pi^2 |f_\xi|^2 |\xi|^4.$$
 For any $x\in\mathbb R^n\setminus\{(0,\cdots,0)\}$, it follows from the spanning property of $\mathcal{A}(f)$ in Lemma \ref{lemma-iso condition} that 
$$x^T Cx=\sum_{\xi\in \mathcal{A}^*(f)}4\pi^2 |f_\xi|^2 |\xi|^2\langle \xi,x\rangle^2>0.$$
So $C$ is positive definite. Consequently, by the AM-GM inequality, 
we have
\begin{equation}\label{ineq-AM}
    \operatorname{tr}C=\lambda_1+\cdots+\lambda_n\geq n(\lambda_1\cdots\lambda_n)^{1/n}=n(\det C)^{1/n},
\end{equation}
where $\lambda_1,\ldots,\lambda_n$ are the eigenvalues of $C$. 
Therefore, it suffices to prove
\begin{equation}\label{ineq-detC}
    \det C\geq \operatorname{Vol}(T^n_V)^{-2}.
\end{equation}

For any given finite subset $F=\{\xi_1,\ldots,\xi_m\}\subset \mathcal{A}^*(f)\subset \Lambda^*_V$, set
$$C_F\triangleq\sum_{\alpha=1}^m4\pi^2 |f_{\xi_\alpha}|^2|\xi_\alpha|^2\,\xi_\alpha\xi_\alpha^T,
\qquad Q_F\triangleq\sum_{\alpha=1}^m4\pi^2 |f_{\xi_\alpha}|^2\xi_\alpha\xi_\alpha^T,$$
$${D}_F=2\pi\Big{(}|f_{\xi_1}|\,\xi_1,~ |f_{\xi_2}|\,\xi_2,~ \cdots, ~|f_{\xi_m}|\,\xi_m\Big),~~~~~\widetilde{D}_F=2\pi\Big{(}|f_{\xi_1}||\xi_1|\,\xi_1,~ |f_{\xi_2}||\xi_2|\,\xi_2,~ \cdots, ~|f_{\xi_m}||\xi_m|\,\xi_m\Big).$$
Note that there hold 
$$D_FD_F^T=Q_F,~~~~~~\widetilde{D}_F\widetilde{D}_F^T=C_F.$$ 
By the Cauchy-Binet formula, we obtain
\begin{equation}\label{eq-QF}
\det Q_F=\sum_{
1\leq \alpha_1<\cdots<\alpha_n\leq m}\;\prod_{j=1}^n\left(4\pi^2|f_{\xi_{\alpha{_j}}}|^2\right)\det(\xi_{\alpha_1},\ldots,\xi_{\alpha_n})^2.
\end{equation}
\begin{equation}\label{eq-CF}
 \det C_F=\sum_{
 1\leq \alpha_1<\cdots<\alpha_n\leq m}\;
 \prod_{j=1}^n\left(4\pi^2 |f_{\xi{_{\alpha{_j}}}}|^2|\xi_{\alpha{_j}}|^2\right)
\det(\xi_{\alpha_1},\ldots,\xi_{\alpha_n})^2,
\end{equation}
It follows from Hadamard's inequality that 
\begin{equation}\label{ineq-Har-V}
    \prod_{j=1}^n|\xi_{\alpha_j}|\geq|\det(\xi_{\alpha_1},\ldots,\xi_{\alpha_n})|
\geq\det(\Lambda^*_V)=\operatorname{Vol}(T^n_V)^{-1}.
\end{equation}
Substituting \eqref{ineq-Har-V} into \eqref{eq-CF}, together with \eqref{eq-QF}, we obtain
$$
\det C_F\geq \operatorname{Vol}(T^n_V)^{-2}\det Q_F.
$$

Since $C - C_F$ is positive semi-definite, Lemma \ref{lemma-matrix} yields
\[
\det C \geq \det C_F.
\]
Consequently, for every finite subset $F \subset \mathcal{A}^*(f)$, we have
\[
\det C \geq \operatorname{Vol}(T^n_V)^{-2} \det Q_F.
\]
Now choose an exhaustive sequence of finite subsets $\{F_k\}$ of $\mathcal{A}^*(f)$.
 It follows from \eqref{eq-int-iso} that
\[
\lim_{k \to +\infty} Q_{F_k} = I_n.
\]
Moreover, since $f$ is smooth, the defining series for $C$ converges absolutely, and hence $C_{F_k}\rightarrow C$ as $k\rightarrow\infty$.
By the continuity of the determinant, we obtain
\[
\det C \geq \operatorname{Vol}(T^n_V)^{-2} \lim_{k \to +\infty} \det Q_{F_k} = \operatorname{Vol}(T^n_V)^{-2}.
\]


We now consider the equality case. 

{\bf Claim.} {\em If $\det C = \operatorname{Vol}(T^n_V)^{-2}$, then for any linearly independent vectors $\xi_{1},\ldots,\xi_{n}\in\mathcal A^*(f)$, we have 
\begin{equation}\label{eq-equality}
\prod_{j=1}^n|\xi_{j}|=|\det(\xi_{1},\ldots,\xi_{n})|
=\det(\Lambda^*_V)=\operatorname{Vol}(T^n_V)^{-1}.
\end{equation}}

We prove the claim by contradiction. Suppose that there exists a linearly independent set $\{\tau_1,\ldots,\tau_n\}\subset \mathcal A^*(f)$ such that
$$\delta\triangleq\prod_{j=1}^n|\tau_j|^2-\operatorname{Vol}(T^n_V)^{-2}>0.$$
Then we have
\[
\prod_{j=1}^n 4\pi^2 |f_{\tau_j}|^2 \det(\tau_1, \ldots, \tau_n)^2 \, \delta > 0.
\]
Choose $k_0$ sufficiently large such that $\{\tau_1, \ldots, \tau_n\} \subset F_k$ for all $k \geq k_0$. This implies that
\[
\det C_{F_k} - \operatorname{Vol}(T^n_V)^{-2} \det Q_{F_k}
\geq \prod_{j=1}^n 4\pi^2 |f_{\tau_j}|^2 \det(\tau_1, \ldots, \tau_n)^2 \, \delta > 0,
\qquad k \geq k_0.
\]
This contradicts the limit
\[
\lim_{k \to +\infty} \bigl( \det C_{F_k} - \operatorname{Vol}(T^n_V)^{-2} \det Q_{F_k} \bigr)
= \det C - \operatorname{Vol}(T^n_V)^{-2} = 0.
\]
Therefore, the claim follows. 

Observe that in \eqref{eq-equality}, the first equality implies that $\xi_1, \ldots, \xi_n$ are mutually orthogonal, and the second implies that  $\{\xi_1, \ldots, \xi_n\}$  generate $\Lambda_V^*$. 

By Lemma~\ref{lemma-iso condition}, there exists at least one set of linearly independent vectors $\{\eta_1, \ldots, \eta_n\} \subset \mathcal{A}^*(f)$ such that \eqref{eq-equality} holds. 
Then for any  vector $\xi\in \mathcal{A}^*(f)$, we can write 
$$\xi=\sum_{i=1}^n m_i\eta_i,\qquad m_i\in\mathbb Z.$$
Without loss of generality, assume that $m_1 \neq 0$. Since $\{\xi, \eta_2, \dots, \eta_n\}$ also forms a linearly independent set in $\mathcal{A}^*(f)$, the vectors $\xi, \eta_2, \dots, \eta_n$ are mutually orthogonal and generate $\Lambda_V^*$. 
Hence, $\xi = \pm \eta_1$, and we conclude that
\[
\mathcal{A}^*(f) = \{\pm\eta_1, \dots, \pm\eta_n\}.
\]
Up to an orthogonal transformation, we may assume that each $\eta_j$ is parallel to the standard basis vector $e_j$. 
Then, using the isometric condition
\[
I_n = \sum_{i=1}^n 4\pi^2 \bigl( |f_{\eta_i}|^2 + |f_{-\eta_i}|^2 \bigr) \eta_i \eta_i^T,
\]
we obtain
\[
4\pi^2 \bigl( |f_{\eta_i}|^2 + |f_{-\eta_i}|^2 \bigr) |\eta_i|^2 = 1, \qquad 1 \leq i \leq n.
\]
Consequently,
\[
C = \sum_{i=1}^n 4\pi^2 \bigl( |f_{\eta_i}|^2 + |f_{-\eta_i}|^2 \bigr) |\eta_i|^2 \eta_i \eta_i^T = \sum_{i=1}^n \eta_i \eta_i^T
\]
is similar to $\operatorname{diag}\{ |\eta_1|^2, |\eta_2|^2, \dots, |\eta_n|^2 \}$.
Thus, $|\eta_1|^2, \dots, |\eta_n|^2$ are the eigenvalues of $C$. That these eigenvalues are all equal follows from the equality case characterized in {\eqref{ineq-AM}}. 
\end{proof}
\begin{proof}[Proof of Theorem \ref{thm-forall-n}]
 By M\"{o}bius invariance of the Willmore functional, after replacing $f$ by $\varphi\circ f$, we may assume that $f$ is isometric. Up to a translation, one can assume that the center of gravity of $T_V^n$ is at the origin of $\R^N$. Combining  \eqref{eq-Holder}, \eqref{eq-H} and \eqref{ineq-HV-estimate}, we obtain 
\begin{equation}\label{ineq-Chen-conj}
   \begin{split}
       \int_{T^n_V}|H|^n\mathrm{d}V
       \geq \left(\operatorname{Vol}(T^n_V)\right)^{1-\frac{n}{2}}\left(\frac{4\pi^2}{n}\operatorname{Vol}(T_V^n)^{\frac{n-2}{n}}\right)^\frac{n}{2}=\left(\frac{4\pi^2}{n}\right)^{\frac{n}{2}}.
   \end{split}
\end{equation}
Together with (\ref{eq-W-iso}), it follows that
$$\W(f)\geq \left(4n\pi^2\right)^\frac{n}{2}.$$

We now discuss the equality case, which requires that equality holds simultaneously in \eqref{eq-Holder} and \eqref{ineq-HV-estimate}. This is equivalent to the condition that $|H|$ is constant and
\[
\Lambda_V^* = \operatorname{Span}_\mathbb{Z}\{\eta_1, \eta_2, \dots, \eta_n\}, \qquad \mathcal{A}^*(f) = \{\pm\eta_1, \dots, \pm\eta_n\},
\]
where $\eta_1, \dots, \eta_n \in \mathbb{R}^n$ are mutually orthogonal vectors of the same length. Denote $|\eta_i|=\ell$. Then $(T_V^n,g_0)$ is isometric to
$\mathbb R^n/(\ell^{-1}\mathbb Z)^n$. Since the Willmore energy is invariant under homotheties, we may assume without loss of generality that $\ell=1$. Then the conclusion follows from Lemma~\ref{lem-Cliff}. 
\end{proof}

For a minimal flat $n$-torus $f:T^n_V\rightarrow \mathbb{S}^N$, regarded as a submanifold of $\mathbb{R}^{N+1}$, we have $|H|=1$ and the Willmore functional satisfies $\mathcal{W}(f)=n^n \mathcal{V}(f)$. Applying Theorem~\ref{thm-forall-n} to this case yields the following sharp lower bound for $\mathcal{V}(f)$.   
\begin{corollary}
    \label{thm-mini-iso}
For a minimal isometric immersion $f:(T^n_V,e^{2\rho_0}g_0)\rightarrow \mathbb{S}^N$, its volume satisfies
\begin{equation}
    \label{eq:V}
\mathcal{V}(f)\ge \left(\frac{2\pi}{\sqrt n}\right)^n,     
\end{equation}
with equality holding if and only if 
$f$ is congruent to the Clifford $n$-torus in $\mathbb{S}^{2n-1}$. 
\end{corollary}
  \begin{proof}
It is a  corollary of Theorem~\ref{thm-forall-n}. Here we include a  simple proof. Consider the Fourier expansion of $f$ in (\ref{eq-FS}). Since $f$ is a minimal immersion into the unit sphere $\mathbb{S}^N$, by Takahashi theorem, we have
$$\sum_{\xi\in\Lambda_V^*}\bigl(-4\pi^2|\xi|^2e^{-2\rho_0}+n\bigr)f_\xi e^{2\pi i\langle \xi,x\rangle}=0.$$
By the linear independence of the Fourier modes, for every $\xi\in\Lambda_V^*$, it follows that
$$\bigl(-4\pi^2|\xi|^2e^{-2\rho_0}+n\bigr)f_\xi=0.$$
Therefore, if $f_\xi\neq 0$, then $4\pi^2|\xi|^2=ne^{2\rho_0}$, that is, 
$|\xi|=\frac{e^{\rho_0}\sqrt n}{2\pi}.$
By Lemma \ref{lemma-iso condition}, we can choose $n$ linearly independent frequencies
$\eta_1,\ldots,\eta_n\in\mathcal{A}(f)$, such that $|\eta_j|=\frac{e^{\rho_0}\sqrt n}{2\pi}.$
Since $\eta_1,\ldots,\eta_n$ are elements of the lattice $\Lambda_V^*$, they generate a sublattice of $\Lambda_V^*$. Hence
\begin{equation}\label{eq-lattice}
    \det(V^*)\le |\det(\eta_1,\ldots,\eta_n)|\le|\eta_1|\cdots|\eta_n|=\left(\frac{e^{\rho_0}\sqrt n}{2\pi}\right)^n.
\end{equation}
Here $V^*=(V^{-1})^T$ and the second inequality follows from Hadamard's inequality.  Since $f$ is isometric, we obtain
$\mathcal{V}(f)=\int_{T^n_V}\mathrm{d}V=e^{n\rho_0}\operatorname{Vol}(T^n_V).$ 
Since $\operatorname{Vol}(T^n_V)=\det(V)={\det(V^*)}^{-1}$,  we get \eqref{eq:V}.
The equality case is the same as before and the proof is finished.  \end{proof}
    
\begin{remark}\rm
 Recently, in  \cite{L-W-X2025}, L\"{u}-Wang-Xie investigated which flat $n$-tori admit minimal isometric immersions into spheres and showed that there are numerous minimal flat $n$-tori in $\mathbb{S}^N$. 
This indicates that the lower bound \eqref{eq:V} for the volume of minimal flat $n$-tori is non-trivial; in particular, it characterizes the Clifford $n$-torus in $\mathbb{S}^{2n-1}$ among all minimal flat $n$-tori. 
\end{remark}

\begin{remark}\rm
For surfaces, Kusner's conjecture asserts that the Willmore energy of $\mathbb{R}P^2$ is the smallest among all surfaces except $\mathbb{S}^2$ \cite{K-1996}. In higher-dimensional submanifold theory, however, it is not known which topological type should realize the least Willmore energy among those different from $\mathbb{S}^n$. Guo-Li-Wang proposed the following conjecture in \cite{GuoLiWang2001}.
\begin{conjecture}[{\rm Guo-Li-Wang}]
 Let $M_k$ be an $n$-dimensional closed manifold homeomorphic to $\SB^k\times \SB^{n-k}$. If $f:M_k\to \SB^{n+1}$ is an embedding, then
$$\mathcal W(f)\ge B_{n,k}\triangleq\frac{4\pi^{\frac{n+2}{2}}(n-k)^{\frac{k}{2}}k^{\frac{n-k}{2}}}{n^{\frac n2-2}(n-1)\Gamma\left(\frac{k+1}{2}\right)\Gamma\left(\frac{n-k+1}{2}\right)}.$$
\end{conjecture}

This conjecture provides a lower bound for the Willmore energy of $\mathbb{S}^k\times \mathbb{S}^{n-k}$-type submanifolds. It is natural to compare it with the flat torus case.
Set $C_n:=(4n\pi^2)^{n/2}$. For $n=2$, we have $B_{2,1}=C_2=8\pi^2$. For $n\ge 3$, however, 
\[
B_{n,k}<C_n \qquad (1\le k\le n-1).
\]
Hence, in higher dimensions, the Willmore lower bound for the $n$-torus might be strictly larger than that for the $\mathbb{S}^k\times \mathbb{S}^{n-k}$ case.  We refer to \cite{L-2026-1,Pinkall1986} for some further discussions concerning this problem.

\end{remark}

\section{Li--Yau-Type Estimates for Flat $n$-Tori with Korkine--Zolotarev reduction}\label{section-KZ}

In Li-Yau's proof of the Willmore conjecture for flat $2$-tori, a key point can be stated as an algebraic inequality of lattice vectors. In this section, we generalize this inequality to lattices of rank $n = 3, 4, 5$,  thereby providing an alternative proof of Theorem~\ref{thm-forall-n} for flat $n$-tori when $3\leq n\leq 5$. However, this approach does not provide a sharp inequality when $n \geq 6$, as the algebraic inequality fails in those ranks.

\subsection{The Korkine-Zolotarev reduced lattice bases}
~

We first recall the Korkine-Zolotarev lattice basis reduction; see \cite{Korkine--Zolotarev1873,M-G,P2007} for further details.

Consider the lattice $\Lambda_R=R\mathbb{Z}^n$ generated by the  column vectors $\{R_1, \dots, R_n\}$ of $R$. Denote by $R_i^\perp$ the orthogonal projection of $R_i$ onto $\mathrm{Span}\{R_1, \dots, R_{i-1}\}^\perp$. 
The basis $\{R_1, \dots, R_n\}$ is said to be Korkine-Zolotarev reduced if and only if 
\begin{itemize}
    \item[(1)] for any $1 \leq i < j \leq n$, the orthogonal projection of $R_j$ onto the direction of $R_i^\perp$ has a length not exceeding $\frac{1}{2}\|R_i^\perp\|$ (i.e., ${|\langle R_j, R_i^\perp \rangle|} \leq \frac{1}{2}{\|R_i^\perp\|^2}$); 
    \item[(2)] $R_1$ is one of the shortest non-zero vectors of $\Lambda_R$;
    \item[(3)] for any $2\leq i\leq n$, the projected vector $R_i^\perp$ is one of the shortest non-zero vectors in the orthogonal projection of $\Lambda_R$ onto $\mathrm{Span}\{R_1, \dots, R_{i-1}\}^\perp$. 
\end{itemize} 
By this geometric characterization and the properties of Gram-Schmidt orthogonalization, it is well known that every lattice admits a Korkine-Zolotarev reduced basis.

In what follows, we always assume that the column vectors of $A$ form a Korkine-Zolotarev reduced basis of the lattice $\Lambda_A$. Note that, up to an orthogonal transformation and a dilation, we may assume that $A = (a_{ij})_{1 \leq i, j \leq n}$ is an upper triangular matrix whose diagonal entries are
\begin{equation}\label{eq-inv}
a_{11} = |A_1| = 1, \qquad a_{ii} = |A_i^\perp|, \quad 2 \leq i \leq n.
\end{equation}
Moreover, $A$ satisfies the following two conditions,
\begin{equation}\label{ineq-rij}
    |a_{ij}|\leq \frac{1}{2}a_{ii}, \quad 1\leq i < j\leq n,
\end{equation}
and
\begin{equation}\label{ineq-rii}
    \sum_{k=i}^{n}\left(\sum_{j=k}^{n}m_j a_{kj}\right)^2\geq a_{ii}^2, \quad 1\leq i\leq n, 
\end{equation}
for any $(m_i,\dots,m_n)^T\in\mathbb Z^{n-i+1}\setminus\{\vec{0}\}$. 
From the definition, the following lemma follows immediately. 
\begin{lemma}\cite{Korkine--Zolotarev1873}
The diagonal entries of $A$ satisfy
\begin{equation}\label{eq-Korkine--Zolotarev-1}
    a^2_{i+1,i+1}\geq\frac{3}{4}a^2_{ii},\qquad 1\leq i\leq n-1,
\end{equation}
and
\begin{equation}\label{eq-Korkine--Zolotarev-2}
     a^2_{i+2,i+2}\geq\frac{2}{3}a^2_{ii},\qquad 1\leq i\leq n-2.
\end{equation}
\end{lemma}
Now we state the key algebraic inequality established by Li-Yau in \cite{Li-Yau} for flat $2$-tori. 
\begin{lemma}\label{le-ly}
    Let $(T_A^2,g_0)$ be a flat $2$-torus. For any $\xi=(\xi_1, \xi_2)^T\in{\Lambda}^*_A\setminus\{\vec{0}\}$, we have
    \begin{equation}\label{eq-point-estimate-ly}
        |\xi|^4\geq\xi_1^2+\frac{\xi_2^2}{a^2_{22}},
    \end{equation}
where $a_{22}$ is the invariant of $\Lambda_A$ defined in \eqref{eq-inv}. 
The equality holds if and only if 
$\xi=\pm(1,0)$ with $a_{12}=0$ or $\xi=\pm(0,\frac{1}{a_{22}})$. 
\end{lemma}

\vspace{1mm}
 
\subsection{Generalization of Li-Yau's algebraic inequality to lattices of rank 
$n\in\{3,4,5\}$}
~

\begin{proposition}\label{lemma-point-estimate}
    Let $\Lambda_A = A\mathbb{Z}^n$ be a lattice of rank $n \in \{3, 4, 5\}$, where the  column vectors of $A$ form a Korkine-Zolotarev reduced basis. For any $\xi=(\xi_1,\ldots,\xi_n)^T \in \Lambda^*_A\setminus\{\vec{0}\}$, we have
\begin{equation}\label{eq-point-estimate}
    \|\xi\|^4 \geq \sum_{i=1}^n \frac{\xi_i^2}{a^2_{ii}},
\end{equation}
where the $a_{ii}$ denote the invariants defined in \eqref{eq-inv}. Furthermore, for a fixed $\xi\in\Lambda_A^*\setminus\{0\}$, equality holds if and only if
$\xi=\pm\frac{e_i}{a_{ii}}$ for some $i\in\{1,\ldots,n\}$ such that $a_{ij}=0, j>i$. In particular, if equality holds for a collection of vectors spanning
$\mathbb R^n$, then $A$ is diagonal.
\end{proposition}
\begin{proof}
It follows from the definition that $\xi\in\Lambda^*_A$ if and only if $A^T\xi=k$ for some $k=(k_1,\ldots,k_n)^T\in\mathbb Z^n$, i.e.,  
\begin{equation}\label{eq-intergal}
    \sum_{i=1}^{j}a_{ij}\xi_i=k_j\in\mathbb Z,
\qquad 1\leq j\leq n.
\end{equation}

    For any $2 \leq m \leq n$, let $\zeta_m = \sum_{i=1}^m \xi_i^2$. We will prove by induction on $m$ that 
    \begin{equation}\label{eq-induction-hypothesis}
        \zeta_m^2 \geq \sum_{i=1}^m \frac{\xi_i^2}{a_{ii}^2}
    \end{equation}
    holds for all $2 \leq m \leq 5$, where the base case $m=2$ is the algebraic inequality \eqref{eq-point-estimate-ly} established by Li-Yau. 

    Assume the inequality holds for $m-1$ (where $3 \leq m \leq 5$). We now consider the case for $m$. Observe that
    \begin{equation}\label{eq-zeta-expansion}
        \zeta_m^2 = \left(\zeta_{m-1} + \xi_m^2\right)^2 = \zeta_{m-1}^2 + \left(2\zeta_{m-1} + \xi_m^2\right)\xi_m^2.
    \end{equation}
    By the induction hypothesis, 
    it suffices to show that
    \begin{equation}\label{eq-suffices-to-show}
        2\zeta_{m-1} + \xi_m^2 \geq \frac{1}{a_{mm}^2}, \quad \text{whenever } \xi_m \neq 0.
    \end{equation}
    We analyze this in three exhaustive sub-cases based on the non-zero components of $\xi$:
    
    \textbf{Case 1: $\zeta_{m-1} = 0$.} Then it follows from \eqref{eq-intergal} that 
    $\xi_m = \frac{k_m}{a_{mm}}$ for some non-zero integer $k_m$. Consequently,
    $$2\zeta_{m-1} + \xi_m^2 = \frac{k_m^2}{a_{mm}^2} \geq \frac{1}{a_{mm}^2}.$$
    
    \textbf{Case 2: $\zeta_{m-1} \neq 0$ but $\zeta_2 = 0$.} 
    This case only occurs when $m \in \{4, 5\}$. Let $l$ be the first index in $\{3,\cdots,  m-1\}$ such that $\xi_l = \frac{k_l}{a_{ll}} \neq 0$ for some integer $k_l$. By the Korkine-Zolotarev properties \eqref{eq-Korkine--Zolotarev-1} and \eqref{eq-Korkine--Zolotarev-2}, we can bound $a_{ll}^2$ relative to $a_{mm}^2$. Specifically, since $m-l\leq 5-3 = 2$, we have $a_{mm}^2 \geq \frac{2}{3}a_{ll}^2$ (by \eqref{eq-Korkine--Zolotarev-2} if $m-l=2$) or $a_{mm}^2 \geq \frac{3}{4}a_{ll}^2$ (by \eqref{eq-Korkine--Zolotarev-1} if $m-l=1$). In either case, $a_{mm}^2 \geq \frac{2}{3}a_{ll}^2$. Hence,
    $$2\zeta_{m-1} + \xi_m^2 > 2\xi_l^2 = \frac{2k_l^2}{a_{ll}^2} \geq \frac{2}{a_{ll}^2} \geq \frac{4}{3a_{mm}^2} \geq \frac{1}{a_{mm}^2}.$$
    
    \textbf{Case 3: $\zeta_2 \neq 0$.}
    In this case, we first show that 
    \begin{equation}\label{eq-xi12-estimate}
       \zeta_2= \xi_1^2+\xi^2_2\geq\frac{1}{a^2_{22}}.
    \end{equation}
    In fact, from \eqref{eq-intergal} and $a_{11}=1$, there exists $k_1,k_2\in\mathbb{Z}$ such that  
 $$\xi_1=k_1\qquad \xi_2=\frac{k_2-a_{12}k_1}{a_{22}}.$$
By (\ref{ineq-rii}), we have $a_{12}^2+a_{22}^2\geq a_{11}^2=1$. Hence
$$ \xi_1^2+\xi_2^2
=
k_1^2+\frac{(k_2-a_{12}k_1)^2}{a_{22}^2}
\geq
\frac{k_1^2(1-a_{12}^2)+(k_2-a_{12}k_1)^2}{a_{22}^2}\geq
\frac{k_1^2+k_2^2-2a_{12}k_1k_2}{a_{22}^2}.$$
Since $(\xi_1,\xi_2)\neq (0,0)$, we have $(k_1,k_2)\neq (0,0)$. Then \eqref{eq-xi12-estimate} follows from 
(\ref{ineq-rij}). 
    On the other hand, since $3\leq m \leq 5$, 
    it follows from \eqref{eq-Korkine--Zolotarev-1} and \eqref{eq-Korkine--Zolotarev-2} that 
    $$a_{mm}^2 \geq \frac{2}{3}a_{33}^2 \geq \frac{2}{3}\left(\frac{3}{4}a_{22}^2\right) = \frac{1}{2}a_{22}^2.$$ Then we obtain 
    $$2\zeta_{m-1} + \xi_m^2 > 2\zeta_2 \geq \frac{2}{a_{22}^2} \geq \frac{1}{a_{mm}^2}.$$
    
    Therefore, \eqref{eq-suffices-to-show} holds in all cases. Setting $m=n$ completes the proof of the inequality \eqref{eq-point-estimate}.

    We now characterize the equality cases. The expansion \eqref{eq-zeta-expansion} implies that equality in \eqref{eq-point-estimate} requires $(2\zeta_{n-1} + \xi_n^2)\xi_n^2 = \frac{\xi_n^2}{a_{nn}^2}$. Since we established that $2\zeta_{n-1} + \xi_n^2 > \frac{1}{a_{nn}^2}$ whenever $\zeta_{n-1} \neq 0$ and $\xi_n\neq 0$, equality can only occur if either $\xi_n = 0$ or $\zeta_{n-1} = 0$.

    If $\zeta_{n-1} = 0$, then $\xi = (0, \dots, 0, \xi_n)^T$. The equation $\|\xi\|^4 = \frac{\xi_n^2}{a_{nn}^2}$ simplifies to $\xi_n^4 = \frac{\xi_n^2}{a_{nn}^2}$. Since $\xi \neq 0$, this yields $\xi_n = \pm \frac{1}{a_{nn}}$, meaning $\xi = \pm \frac{e_n}{a_{nn}}$.

    If $\xi_n = 0$, the equality conditions reduce strictly to those in dimension $n-1$. Applying the same logic recursively, we conclude that $\xi$ must take the form $\xi = \pm a_{ii}^{-1}e_i$ for some $1 \leq i \leq n$.

    Finally, for a fixed index $i$, the condition $\xi = \pm a_{ii}^{-1}e_i \in \Lambda_A^*$ is equivalent to 
    $$A^T(a_{ii}^{-1}e_i) = \left(0, \dots, 0, 1, \frac{a_{i,i+1}}{a_{ii}}, \dots, \frac{a_{i,n}}{a_{ii}}\right)^T\in \mathbb{Z}^n.$$
    Consequently, we must have $\frac{a_{ij}}{a_{ii}} \in \mathbb{Z}$ for all $j > i$. However, the Korkine-Zolotarev reduction condition bounds the off-diagonal entries as $\left| \frac{a_{ij}}{a_{ii}} \right| \leq \frac{1}{2}$. The only integer satisfying this bound is zero. Thus, $a_{ij} = 0$ for all $j > i$, completing the proof.
\end{proof}

\begin{theorem}\label{thm-iso}
 Let $f:(T_A^n,g_0)\longrightarrow \mathbb{R}^N ~(2\leq n\leq 5)$ be  an isometric immersion with $A$ as a Korkine-Zolotarev reduction of the corresponding lattice $\Lambda_A$ formed in \eqref{eq-inv}, then its Willmore energy $\W(f)$ satisfies
   $$\mathcal{W}(f)\geq \left(2\pi\right)^{n}\left(\prod_{i}a_{ii}\right)\left(\sum_i\frac{1}{a_{ii}^2}\right)^{\frac{n}{2}},$$
with equality holding if and only if  $A$ is diagonal and $f$ is congruent to
\begin{equation}\label{eq-f-iso-express}
    \begin{split}
        f(x_1,\ldots,x_n)
=\frac{1}{2\pi}\left(
a_{11}\cos\frac{2\pi x_1}{a_{11}} ,
a_{11}\sin\frac{2\pi x_1}{a_{11}},
\ldots,
a_{nn}\cos\frac{2\pi x_n}{a_{nn}} ,a_{nn}\sin\frac{2\pi x_n}{a_{nn}}
\right).
    \end{split}
\end{equation}
\end{theorem}
When $n=2$, this is precisely Proposition~2 of Li-Yau in \cite{Li-Yau}.
  \begin{proof}
After a translation, we can assume that the center of gravity of $T_A^n$ is at the origin of $\R^N$.
By (\ref{eq-H}) and Proposition \ref{lemma-point-estimate} for $3\geq n\geq 5$, and Lemma \ref{le-ly} for $n=2$, we obtain we obtain
   \begin{equation}\label{eq-H-estimate}
\begin{split}
    \int_{T^n_A}|H|^2\mathrm{d}V&\geq \frac{16}{n^2}\pi^4\operatorname{Vol}(T^n_A)\sum_{\xi\in\Lambda_A^*\setminus\{\vec{0}\}}\left(|f_{\xi}|^2\sum_{i}\frac{\xi_i^2}{a^2_{ii}}\right)\\
    &=\frac{16}{n^2}\pi^4\operatorname{Vol}(T^n_A)\sum_{\xi\in\mathcal{A}^*(f)}\left(|f_{\xi}|^2\sum_{i}\frac{\xi_i^2}{a^2_{ii}}\right).
\end{split}        
   \end{equation}
Since $f$ is an isometric immersion, setting $i=j$ in (\ref{eq-int-iso}) yields 
$$4\pi^2\sum_{\xi\in\mathcal{A}^*(f)}|f_\xi|^2\xi_i^2=1.$$
Substituting this identity into (\ref{eq-H-estimate}), we have
\begin{equation}\label{eq-H-estimate2}
     \int_{T^n_A}|H|^2\mathrm{d}V\geq\frac{4}{n^2}\pi^2\operatorname{Vol}(T^n_A)\left(\sum_i\frac{1}{a_{ii}^2}\right).
\end{equation}
Combining (\ref{eq-W-iso}),  (\ref{eq-Holder})  and (\ref{eq-H-estimate2}), we obtain
\begin{equation*}
    \begin{split}
        \W(f)&\geq n^n\left(\operatorname{Vol}(T^n_A)\right)^{1-\frac{n}{2}}\left(\frac{4}{n^2}\pi^2\operatorname{Vol}(T^n_A)\left(\sum_i\frac{1}{a_{ii}^2}\right)\right)^{\frac{n}{2}}\\
        &={\left(4\pi^2\right)^{\frac{n}{2}}}\operatorname{Vol}(T^n_A)\left(\sum_i\frac{1}{a_{ii}^2}\right)^{\frac{n}{2}}\\
        &={\left(4\pi^2\right)^{\frac{n}{2}}}\left(\prod_{i}a_{ii}\right)\left(\sum_i\frac{1}{a_{ii}^2}\right)^{\frac{n}{2}}.
    \end{split}
\end{equation*}

Equality in \eqref{eq-H-estimate}, together with Proposition  \ref{lemma-point-estimate} (and Lemma \ref{le-ly} when $n=2$), implies that every frequency in $\mathcal A^*(f)$ is of the form $\pm a_{ii}^{-1}e_i$, with $a_{ij}=0$ for all $j>i$. Since
$\mathcal A^*(f)$  spans $\R^n$, the matrix $A$ is diagonal, and the isometric condition then yields \eqref{eq-f-iso-express}; the converse follows by direct computation.
\end{proof}
\begin{remark}\rm
When  $3\leq n\leq5$, Theorem \ref{thm-iso}
 provides a finer lower bound  than the one in Theorem \ref{thm-forall-n}. 

\end{remark}
\begin{remark}\label{rk-4.6}\rm
{ 
When $n\ge6$, the following examples show that the algebraic estimate (\ref{eq-point-estimate}) can not hold for Korkine--Zolotarev reduced lattices.
   When $n=6$, the corresponding lattice is generated by   column vectors of 
   \begin{equation*}
A_6=\begin{pmatrix}
1&0&0&0&0&\frac18\\
0&1&-\frac13&-\frac13&-\frac13&-\frac13\\
0&0&\frac{2\sqrt2}{3}&-\frac{\sqrt2}{3}&\frac{7\sqrt2}{24}&-\frac{\sqrt2}{3}\\
0&0&0&\frac{\sqrt6}{3}&-\frac{\sqrt6}{8}&-\frac{\sqrt6}{12}\\
0&0&0&0&\frac{\sqrt{10}}{4}&-\frac{\sqrt{10}}{8}\\
0&0&0&0&0&\frac{\sqrt{30}}{8}
\end{pmatrix}.
   \end{equation*}
The lower-right $5\times5$ block of $A_6$ corresponds to the
Korkine--Zolotarev reduced form $E_{5b}$ in \cite[Table~6.1]{P2007}. Hence,  it is direct to check that $A_6$ is Korkine-Zolotarev reduced. By taking $\xi=(1,0,0,0,0,-1/\sqrt{30})^T\in (A_6^{-1})^{T}\mathbb{Z}^6$, we have
$$|\xi|^4=\frac{961}{900}<\frac{964}{900}=\sum_{i=1}^6\frac{\xi_i^2}{a_{ii}^2}.$$
For $n>6$, set
\[
A_n=I_{n-6}\oplus A_6,\qquad
\xi_n=
\left(
\underbrace{0,\ldots,0}_{n-6},
1,0,0,0,0,-\frac1{\sqrt{30}}
\right)^T.
\]
Then one can check that $A_n$ is Korkine--Zolotarev reduced, $\xi_n\in\Lambda_{A_n}^*$, and the same strict inequality holds. This shows that the  estimate \eqref{eq-point-estimate} can not be extended to all Korkine--Zolotarev reduced lattices in dimensions $n\ge6$. A
different global estimate exploiting relations among the full Fourier support may nevertheless hold; this is precisely the role of the argument in Section 3.}
\end{remark}

\section{On Chen’s total mean curvature under M\"obius deformations}\label{Section-chen-conj}

For the flat $n$-tori, the Willmore functional is a constant multiple of the total mean curvature, and hence Chen’s conjecture in this case is an immediate consequence of Theorem  \ref{thm-forall-n}. Although the Willmore functional is M\"obius invariant, the total mean curvature could decrease under  M\"obius transformations, which provides counterexamples to Chen’s conjecture in the general case when $n\ge3$.

\subsection{M\"{o}bius transformation formula of the total mean curvature for submanifolds in $\SB^N\subset\R^{N+1}$ }~

For a submanifold ${M^n}$ of $\mathbb{R}^{N+1}$ which is contained in the unit sphere $\SB^{N}$, the Euclidean total
mean curvature can be written as
$$
\int_{M^n}|H
|^n\mathrm{d}V
=\int_{M^n}\left(1+| H^{\SB^N}|^2\right)^{n/2}\mathrm{d}V .
$$
Here $H$ denotes the mean curvature vector of $M^n$ in $\R^{N+1}$ and $H^{\SB^N}$ denotes the mean curvature vector of $M^n$ in $\SB^{N}$. It is also called $\mathcal H$-functional in \cite{Chen2015, L-2026}, where  minimal submanifolds in $\SB^N$ are shown to be critical submanifolds of this functional. Moreover, L\"{u} computed the second variation of this functional for minimal submanifolds in $\SB^N$, which 
can be expressed by { 
\[
\mathcal H''(0)
=
\frac1n
\int_{M^n}
\langle \mathcal{L}(\mathcal{L}-n)X,X\rangle\,\mathrm d V ,
\]
where $\mathcal L$ denotes the Jacobi operator of the area functional and $X$ is the variational vector field.
See \cite{L-2026} and reference therein for more details.} 
In particular, the absence of eigenvalues of $\mathcal L$ of $M$ in the interval $(-n,0)$ implies the $\mathcal H$-stability of $M^n$.
From the spectral computation carried out in \cite{W-W-2023}, it follows that every Clifford minimal product is stable for the $\mathcal{H}$-functional. Moreover, by Simons' result \cite{Simons}, the vector fields of M\"obius transformations are Jacobi fields of the $\mathcal{H}$-functional.
Since the  total mean curvature 
is not M\"{o}bius invariant when $n\geq3$, it is natural to examine   higher order variations of this functional along M\"obius transformations of $\SB^{N}$. 
It turns out that there exists a family of M\"obius transformations  of $\SB^{N}$ that strictly decreases the total mean curvature. 
This disproves Chen's conjecture.

Let $v=(v_1, v_2, \cdots, v_{N+1})\in\SB^{N}\subset\mathbb R^{N+1}$ be a unit vector and set $a=\varepsilon v$ with $|\varepsilon|<1$.
Consider the M\"obius transformation of $\SB^{N}$ defined by
$$
\Phi_a(y)
=
\frac{(1-|a|^2)(y+a)}{|y+a|^2}
+a,
\qquad y\in \SB^{N},
$$
whose conformal factor is $e^{2\rho_a}(y)=\left(\frac{1-|a|^2}{|y+a|^2}\right)^2$. 
\begin{theorem}\label{prop-H}
Let $f: M^n \rightarrow \SB^N$ be a closed minimal submanifold of dimension $n$. Regarding $f_\varepsilon=\Phi_{\varepsilon v}\circ f$ as a submanifold in $\mathbb{R}^{N+1}$, the total mean curvature satisfies 
\begin{equation}\label{eq-mini-sub-H}
\begin{split}
 \int_{f_\varepsilon({M^n})}
|H_\varepsilon|^n\,\mathrm dV
=&\mathcal{V}(f_0)+ \frac{n(n^2-4)}{12} t^3 \int_{M^n} \langle f,v\rangle^3 \, \mathrm{d}V\\
&- \frac{n(n-2)}{4} t^4 \int_{M^n} \langle f,v\rangle \langle B(v^\top, v^\top), v^\bot \rangle \, \mathrm{d}V
+ O(t^5), ~~~t\ll 1,
\end{split}
\end{equation}
where $H_\varepsilon$ denotes the mean curvature vector of $f_{\varepsilon}$ in $\mathbb{R}^{N+1}$, $t=\frac{2\varepsilon}{1+\varepsilon^2}$ and $B$ is the second fundamental form of ${M^n}$ in $\mathbb{S}^N$. 
\end{theorem}
\begin{proof}
Let $H_\varepsilon^{\SB^N}$ denote the mean curvature vector of $f_\varepsilon$ in $\SB^N$. Since $\Phi_a:(\SB^N, g_0)\to(\SB^N,g_0)$ is a conformal transformation with the conformal factor $e^{2\rho_a}$, 
applying the standard formula yields
\[
H_{\varepsilon}^{\SB^N}=
-e^{-2\rho_a}d\Phi_a\left((\operatorname{grad}_{\mathbb S^{N}}\rho_a)^\perp\right)=\frac{2e^{-2\rho_a}}{|f+a|^2}d\Phi_a\left(\left(a-\langle a,f\rangle f\right)^\perp
\right)=\frac{2e^{-\rho_a}}{1-|a|^2}d\Phi_a(a^\perp
),\]
where $\perp$ denotes orthogonal projection onto the normal bundle of $f$. 
Therefore, we have 
\begin{equation*}
 |H_\varepsilon^{\mathbb S^{N}}|=\frac{2|a^\perp|}{1-|a|^2},\qquad  |H_\varepsilon|^2=1+\frac{4|a^\perp|^2}{(1-|a|^2)^2}=1+\frac{4\varepsilon^2|v^\perp|^2}{(1-\varepsilon^2)^2}.  
\end{equation*}
It follows that
\begin{equation}\label{eq-tay}
    \begin{split}
        \int_{f_\varepsilon({M^n})}
|H_\varepsilon|^n\,\mathrm dV
&=
\int_{M^n}
\left(1+\frac{4\varepsilon^2|v^\perp|^2}{(1-\varepsilon^2)^2}\right)^{\frac n2}
\left(\frac{1-\varepsilon^2}{1+\varepsilon^2+2\varepsilon\langle f,v\rangle}\right)^n
\,\mathrm dV
\\
&=
\int_{M^n}
\frac{\left((1-\varepsilon^2)^2+4\varepsilon^2|v^\perp|^2\right)^{\frac n2}}
{\left(1+\varepsilon^2+2\varepsilon\langle f,v\rangle\right)^n}
\,\mathrm dV\\
&=\int_{M^n}
\frac{\left((1+\varepsilon^2)^2-4\varepsilon^2(|v^\top|^2+\langle f,v\rangle^2)\right)^{\frac n2}}
{\left(1+\varepsilon^2+2\varepsilon\langle f,v\rangle\right)^n}
\,\mathrm dV\\
&=\int_{M^n}
\left(1-\frac{4\varepsilon^2}{(1+\varepsilon^2)^2}\left(|v^\top|^2+\langle f,v\rangle^2\right)\right)^{\frac n2}
{\left(1+\frac{2\varepsilon}{1+\varepsilon^2}\langle f,v\rangle\right)^{-n}}
\,\mathrm dV.
    \end{split}
\end{equation}
Set 
$t=\frac{2\varepsilon}{1+\varepsilon^2}$.  Applying the Taylor expansion to \eqref{eq-tay} and using  Lemma~\ref{eq-minimal}, we get (\ref{eq-mini-sub-H}).

\end{proof}

\subsection{Further discussions on the variational formula}~

Applying Theorem \ref{prop-H} to the  Clifford $n$-torus $\CC$, we obtain the following corollary. 

\begin{corollary}\label{cor-T}
 For the Clifford $n$-torus  $f: T^n_{V_0}\rightarrow \mathbb{S}^{2n-1}$ with $V_0=\frac{2\pi}{\sqrt{n}}I_n$,  
 along the variation $\Phi_{\varepsilon v}$, there holds 
    \begin{equation}\label{eq-H-t}
\begin{split}
   \int_{f_{\varepsilon}(T_{V_0}^n)}
| H_\varepsilon|^n\,\mathrm dV_\varepsilon 
=\mathcal{V}(f_0)
\left(
1+
\frac{n-2}{32}
\bigl((n+1)S_v-2\bigr)t^4
+
O(t^5)
\right),~~~t\ll 1,  
\end{split}
\end{equation}
where $t=\frac{2\varepsilon}{1+\varepsilon^2}$ and $S_v=\sum_{k=1}^{n}\left(v^2_{2k-1}+v^2_{2k}\right)^2$.
\end{corollary}
\begin{proof}
Set
$$
\alpha_i
=
\cos\sqrt n x_i\,E_{2i-1}
+
\sin\sqrt n x_i\,E_{2i},
\qquad
\text{and}\qquad
\beta_i
=
-\sin\sqrt n x_i\,E_{2i-1}
+
\cos\sqrt n x_i\,E_{2i}.
$$
Then
$$
f=\frac{1}{\sqrt n}\sum_{i=1}^n\alpha_i,
\qquad
\frac{\partial f}{\partial x_i}=\beta_i,\qquad 1\leq i\leq n.
$$
Thus, $\{\beta_i\}$ forms an orthonormal tangent frame of the Clifford $n$-torus. A direct computation gives
$$
B(\beta_i,\beta_i)=f-\sqrt n\,\alpha_i,
\qquad \text{and}
\qquad B(\beta_i,\beta_j)=0,
\qquad i\ne j.
$$

For any $v\in\mathbb R^{2n}$, using $\{\alpha_i,\beta_i\}$, we may write
$$
v
=
\sum_{i=1}^n
\left(
\langle\alpha_i,v\rangle\alpha_i
+
\langle\beta_i,v\rangle\beta_i
\right).
$$
Then 
$$
\begin{aligned}
\langle f,v\rangle 
=
\frac{1}{\sqrt n}\sum_{i=1}^n\langle\alpha_i,v\rangle, \quad
v^{\top}=\sum_{i=1}^n\langle\beta_i,v\rangle\beta_i, \quad v^\perp
=v-\langle f,v\rangle f-v^{\top} =
\frac{1}{\sqrt n}\sum_{i=1}^n\langle\alpha_i,v\rangle (\sqrt n\alpha_i-f),
\end{aligned}
$$
and,
$$
\begin{aligned}
B(v^{\top},v^{\top})
=
\sum_{i,j}
\langle\beta_i,v\rangle
\langle\beta_j,v\rangle
B(\beta_i,\beta_j)=
\sum_{i=1}^n
\langle\beta_i,v\rangle^2
\left(f-\sqrt n\,\alpha_i\right).
\end{aligned}
$$
It follows that
$$
\left\langle B(v^{\top},v^{\top}),v^\perp\right\rangle=
-\frac{n-1}{\sqrt n}\sum_{i=1}^n\langle\alpha_i,v\rangle \langle\beta_i,v\rangle^2+\frac{1}{\sqrt n}\sum_{i\neq j}^n\langle\alpha_i,v\rangle \langle\beta_j,v\rangle^2,$$
and then 
$$
\begin{aligned}
\langle f,v\rangle\left\langle B(v^{\top},v^{\top}),v^\perp\right\rangle=&-\frac{n-1}{ n}\sum_{i=1}^n\langle\alpha_i,v\rangle^2 \langle\beta_i,v\rangle^2-\frac{n-2}{n}\sum_{i\neq j}^n\langle\alpha_j,v\rangle \langle\alpha_i,v\rangle \langle\beta_i,v\rangle^2\\
&+\frac{1}{n}\sum_{i\neq j}^n\langle\alpha_i,v\rangle^2 \langle\beta_j,v\rangle^2+\frac{1}{n}\sum_{i\neq j\neq k}^n\langle\alpha_i,v\rangle \langle\alpha_k,v\rangle \langle\beta_j,v\rangle^2
\end{aligned}.$$
By the symmetry, we derive that $\int_{T_{V_0}^n}\langle f,v\rangle ^3\,\mathrm dV=0$,  and 
$$\int_{T_{V_0}^n}\langle f,v\rangle \left\langle B(v^{\top},v^{\top}),v^\perp\right\rangle\,\mathrm dV=-\frac{n-1}{ n}\sum_{i=1}^n\int_{T_{V_0}^n} \langle\alpha_i,v\rangle^2 \langle\beta_i,v\rangle^2\,\mathrm dV+\frac{1}{n}\sum_{i\neq j}^n\int_{T_{V_0}^n}\langle\alpha_i,v\rangle^2 \langle\beta_j,v\rangle^2\,\mathrm dV.$$
Note that 
\begin{equation*}
    \begin{split}
&\langle\alpha_i,v\rangle^2
=
\left(
\cos\sqrt n x_i\,v_{2i-1}
+
\sin\sqrt n x_i\,v_{2i}
\right)^2 
=
\left(v_{2i-1}^2+v_{2i}^2\right)
\cos^2\left(\sqrt n x_i-\phi_i\right),\\
&\langle\beta_i,v\rangle^2
=
\left(
-\sin\sqrt n x_i\,v_{2i-1}
+
\cos\sqrt n x_i\,v_{2i}
\right)^2 
=
\left(v_{2i-1}^2+v_{2i}^2\right)
\sin^2\left(\sqrt n x_i-\phi_i\right),
    \end{split}
\end{equation*}
where $\phi_i\in[0,2\pi)$ is some constant. Then the conclusion follows from 
\[
\int_{T^n_{V_0}} \langle\alpha_i, v\rangle^2 \, \langle\beta_j, v\rangle^2 \,\mathrm dV
=
\begin{cases}
\frac{\mathcal{V}(f_0)}{8}\bigl(v_{2i-1}^2+v_{2i}^2\bigr)^2, &i = j, \\[15pt]
\frac{\mathcal{V}(f_0)}{4}\bigl(v_{2i-1}^2+v_{2i}^2\bigr)
\bigl(v_{2j-1}^2+v_{2j}^2\bigr), 
&i \neq j,
\end{cases}
\]
and 
$$1-
\sum_{i\ne j}
\left(v_{2i-1}^2+v_{2i}^2\right)
\left(v_{2j-1}^2+v_{2j}^2\right)=S_v.$$

\end{proof}

The following corollary shows Chen's conjecture fails for $n\ge 3$. 
\begin{corollary}
  When $n \geq 3$, there exist M\"{o}bius transformations of $\SB^{2n-1}$ that strictly decrease the total mean curvature of the Clifford $n$-torus in $\R^{2n}$.
\end{corollary}
\begin{proof}
{ 
By (\ref{eq-H-t}), the sign of the fourth-order coefficient is determined by $(n+1)S_v-2$. More precisely, the functional $\int_{f_{\varepsilon}(T_{V_0}^n)}
| H_\varepsilon|^n\,dV_\varepsilon $ decreases to the fourth order if
$S_v<\frac{2}{n+1}$ and increases to fourth order if $S_v>\frac{2}{n+1}$. When $S_v=\frac{2}{n+1}$, the fourth-order coefficient vanishes, and the expansion in (5.3) alone does not determine the behavior of the functional.

To exhibit a decreasing direction,} we consider the following variation of the Clifford $n$-torus in $\SB^{2n-1}$, induced by M\"obius transformations. Set
\begin{equation}
    \label{eq-vsharp}
v^{\sharp}:=\frac{1}{\sqrt{n}}(1,0,1,0,\ldots,1,0).
\end{equation}
Then $S_{v^{\sharp}} = \sum_{k=1}^{n} \bigl( v_{2k-1}^2 + v_{2k}^2 \bigr)^2 = \frac{1}{n}.$
It follows that
\[
\frac{n-2}{32} \bigl( (n+1) S_{v^{\sharp}} - 2 \bigr)
= \frac{n-2}{32} \left( \frac{n+1}{n} - 2 \right)
= \frac{(n-2)(1-n)}{32n}.
\]
By Corollary~\ref{cor-T}, when $n \geq 3$, for sufficiently small nonzero $\varepsilon$, the total mean curvature decreases along the variation $\Phi_{\varepsilon v^{\sharp}}$. This implies that the Clifford $n$-torus in $\SB^{2n-1}$ is not a local minimizer. Hence $\bigl(\frac{4\pi^2}{n}\bigr)^{n/2}$ is not the minimum of the total mean curvature among all immersions of $n$-tori in $\R^N$ when $N\geq 2n$.
\end{proof}
{ \begin{remark}\rm
A natural question is the infimum of the total mean curvature over all M\"{o}bius images of the Clifford $n$-torus $\CC$ in $\mathbb S^{2n-1}$. A full treatment of this problem lies beyond the scope of this paper; here we only give some observations.
\begin{itemize}
\item[(1)]By exploiting the $T^n$-invariance of $\CC$, we may reduce the M\"{o}bius parameter $a\in\mathbb B^{2n}\subset\mathbb R^{2n}$ to the form
\[
a=(r_1,0,r_2,0,\dots,r_n,0),
\]
so that the M\"{o}bius images of $\CC$ form an $n$-parameter family of $n$-tori.
\item[(2)] It is easy to verify that the unit vector $v^{\sharp}$ in \eqref{eq-vsharp} minimizes  $S_v$ among all $v\in\SB^{2n-1}$. By \eqref{eq-H-t},  this is also the direction  along which the total mean curvature decreases most rapidly to fourth order. 

\item[(3)] Consider the $1$-parameter family of $n$-tori 
defined by $\Phi_{\varepsilon v^{\sharp}}\circ f$. The monotonicity of this family can be analyzed as follows, and the limit 
suggests a possible non-compactness phenomenon: the infimum of the total mean curvature among all $n$-tori may not be attained by an immersion.
\end{itemize}
\end{remark}}
\begin{proposition}\label{lemma-F4}
Let $f:T_{V_0}^n\rightarrow\mathbb{R}^{2n}$ be the Clifford torus. Consider the M\"obius deformation along the direction
$v^{\sharp}$. The total mean curvature of  $\Phi_{\varepsilon v^{\sharp}}\circ f$ , denoted by  $F_n(t)$ with $t=\frac{2\varepsilon}{1+\varepsilon^2}$, is of the form
\begin{equation}
    \label{eq-Ft}F_n(t)=\int_{T_{V_0}^n}
\left(1-t^2\left(|(v^\sharp)^\top|^2+\langle f,v^\sharp\rangle^2\right)\right)^{\frac n2}
{\left(1+t\langle f,v^\sharp\rangle\right)^{-n}}
\,\mathrm dV.
\end{equation}
In particular, when $n=4$, $F_4(t)$ is strictly decreasing for $t\in[0,1)$.
    \end{proposition}

\begin{proof}
The equation \eqref{eq-Ft} follows directly by the definition of $\Phi_{a}$.   

Next we need to consider its derivative.
To this end, set $s\triangleq\langle f,v^{\sharp}\rangle$.  
Then  
$$F_n(t) =\int_{T_{V_0}^n}
\left(1-t^2\left(|\nabla s|^2+s^2\right)\right)^{\frac n2}
{\left(1+t s\right)^{-n}}
\,\mathrm dV= \int_{T_{V_0}^n} (1-t^2P)^{\frac n2} (1+ts)^{-n} \,\mathrm dV,$$
where $P=|\nabla s|^2+s^2$. 

Differentiating $F_n(t)$ with respect to $t$ 
yields 
$$F_n'(t) = -n \int_{T_{V_0}^n} \left(1-t^2 P\right)^{\frac n2-1} (1+ts)^{-n-1} \left(s + tP\right) \,\mathrm dV.$$
By decomposing the inner expression as $s + tP = s(1+ts) + t|\nabla s|^2$ and invoking the condition $\Delta s = -ns$, we obtain 
$$F_n'(t)=\int_{T_{V_0}^n} \left(1-t^2 P\right)^{\frac n2-1} (1+ts)^{-n} \Delta s \,\mathrm dV-n t\int_{T_{V_0}^n} \left(1-t^2 P\right)^{\frac n2-1} (1+ts)^{-n-1}|\nabla s|^2  \,\mathrm dV.$$ 
Applying the divergence theorem to the Laplacian term, we get
$$\begin{aligned}\int_{T_{V_0}^n} \left(1-t^2 P\right)^{\frac n2-1} (1+ts)^{-n} \Delta s \,\mathrm dV &= - \int_{T_{V_0}^n} \left\langle \nabla \left[ \left(1-t^2 P\right)^{\frac n2-1} (1+ts)^{-n} \right], \nabla s \right\rangle \,\mathrm dV\\
&=\left(\frac n2 - 1\right) t^2 \int_{T_{V_0}^n} \left(1-t^2 P\right)^{\frac n2-2} (1+ts)^{-n} \langle \nabla P, \nabla s \rangle \,\mathrm dV \\
&\quad + n t \int_{T_{V_0}^n} \left(1-t^2 P\right)^{\frac n2-1} (1+ts)^{-n-1} |\nabla s|^2 \,\mathrm dV.
\end{aligned}$$
It follows that  
$$F_n'(t) = \left(\frac n2 - 1\right) t^2 \int_{T^n} \left(1-t^2 P\right)^{\frac n2-2} (1+ts)^{-n} \langle\nabla P, \nabla s\rangle \,\mathrm dV.$$

Set $\gamma_k=\cos\sqrt n x_k.$ Then $\nabla\gamma_1, \dots, \nabla\gamma_n$ are orthogonal to each other, and 
$$\vert{}\nabla \gamma_k\vert{}^2 = n(1-\gamma_k^2),~~~1\leq k\leq n.$$
It follows that $|\nabla s|^2=\frac 1 {n^2}\sum_{k=1}^n\vert{}\nabla \gamma_k\vert{}^2 =\frac 1 {n}\sum_{k=1}^n (1-\gamma_k^2)$, and 
$$\begin{aligned}\langle \nabla P, \nabla s \rangle 
&= \frac{2}{n^2} \sum_{k=1}^n (s-\gamma_k) \vert{}\nabla \gamma_k\vert{}^2=\frac{2}{n^2} \sum_{k=1}^n \sum_{j=1}^n(\gamma_j-\gamma_k)(1-\gamma_k^2)\\
&=\frac{2}{n^2}\sum_{j<k}^n (\gamma_j-\gamma_k)(\gamma_j^2-\gamma_k^2) =\frac{2}{n^2}\sum_{j<k}^n (\gamma_j-\gamma_k)^2(\gamma_j+\gamma_k) \end{aligned}.$$
Hence, we have 
\begin{equation}\label{eq-mobius-Fn-prime-grad}
    F_n'(t) = \frac {n-2}{n^2}  t^2 \sum_{j<k}^n  \int_{T_{V_0}^n} \left(1-t^2 P\right)^{\frac {n-4}{2}} (1+ts)^{-n} (\gamma_j-\gamma_k)^2(\gamma_j+\gamma_k) \,\mathrm dV.
\end{equation}

Observe that for every smooth function $G=G(\gamma_1,\cdots,\gamma_n)$ on $T^n_{V_0}$ , we have
\begin{equation}\label{eq-mobius-gamma-integration}
\int_{T_{V_0}^n}\gamma_iG\,\mathrm dV
=
\int_{T_{V_0}^n}
(1-\gamma_i^2)\frac{\partial G}{\partial\gamma_i}
\,\mathrm dV.
\end{equation}

Applying (\ref{eq-mobius-gamma-integration}) to
(\ref{eq-mobius-Fn-prime-grad}) and simplifying, we obtain
 \begin{equation}\label{eq-mobius-Fn-prime-expanded}
F_n'(t)=
\frac{(n-4)(n-2)}{3n^3}t^4
\sum_{j<k}^n
\int_{T_{V_0}^n}
K_{j,k}
\,\mathrm dV-\frac{n-2}{3n^2}t^3
\sum_{j<k}^n
\int_{T_{V_0}^n}
W_{j,k}\mathrm dV.
\end{equation}
with
\[K_{j,k}:=(1-t^2P)^{\frac {n-6}{2}}
(1+ts)^{-n}
(\gamma_j-\gamma_k)^2\cdot
\left[
(1-\gamma_j^2)(\gamma_j-s)
+
(1-\gamma_k^2)(\gamma_k-s)
\right]\]
and 
\[W_{j,k}:=(1-t^2P)^{\frac {n-4}{2}}
(1+ts)^{-n-1}\cdot
(\gamma_j-\gamma_k)^2
(2-\gamma_j^2-\gamma_k^2)
\,\]
Substituting  $n=4$ into (\ref{eq-mobius-Fn-prime-expanded}), we obtain
$$
F_4'(t)
=
-\frac{t^3}{24}\sum_{j<k}^n
\int_{T_{V_0}^n}
\frac{
(\gamma_j-\gamma_k)^2
(2-\gamma_j^2-\gamma_k^2)
}{
(1+ts)^5
}
\,\mathrm dV.
$$
For every $t\in[0,1)$,
$$
\frac{
(\gamma_j-\gamma_k)^2
(2-\gamma_j^2-\gamma_k^2)
}{
(1+ts)^5
}
\geq0.
$$
Therefore,
$
F_4'(t)<0$  when $ 0<t<1$, that is,  $F_4(t)$ is strictly decreasing on $[0,1)$.
\end{proof}
{
\begin{proposition}
 Under the hypothesis of Proposition~\ref{lemma-F4}, we have
\begin{equation}
\lim_{t\to 1^-} F_n(t)=\omega_n+J_n,
\label{eq:boundary-limit}
\end{equation}
where
$
J_n=\int_{T_{V_0}^n}\widetilde\rho\,\mathrm dV$ and $\widetilde\rho
=
\left(
\frac{\sqrt{1-P}}{1+s}
\right)^n
$.
\end{proposition}

\begin{proof}
Set $
p=
\left(
\frac{\pi}{\sqrt n},\ldots,\frac{\pi}{\sqrt n}
\right),
$ and $
\rho_t(x)
=
\frac{(1-t^2P(x))^{n/2}}{(1+ts(x))^n}$. As $t\to1^-$, one has $\rho_t(x)\to\widetilde\rho(x)$ for every
$x\neq p$, whereas $1+s(p)=0$. We therefore decompose
\[
F_n(t)
=
\int_{T_{V_0}^n\setminus B_r(p)}
\rho_t\,\mathrm dV
+
\int_{B_r(p)}
\rho_t\,\mathrm dV,
\]
where $B_r(p)$ denotes the geodesic ball of radius $r$ centered at $p$
with respect to the flat metric $g_0$ on $T_{V_0}^n$.

We first consider the integral away from $p$. By the computation in
Proposition~\ref{lemma-F4},
\[
1-P
=
\frac{1}{n}\sum_{i=1}^n\gamma_i^2-s^2
=
\frac{1}{n}\sum_{i=1}^n(1+\gamma_i)^2-(1+s)^2.
\]
Since $1+\gamma_i\geq0$, it follows that
\[
\begin{aligned}
1-P
\leq
\frac{1}{n}
\left(
\sum_{i=1}^n(1+\gamma_i)
\right)^2
-(1+s)^2
=
\frac{1}{n}(n+ns)^2-(1+s)^2
=
(n-1)(1+s)^2.
\end{aligned}
\]
Together with $1-P\geq0$, this gives
\[
0\leq\widetilde\rho\leq(n-1)^{n/2}
\qquad
\text{on }T_{V_0}^n\setminus\{p\}.
\]
In particular, $J_n<\infty$ and $
\lim_{r\to0}
\int_{B_r(p)}
\widetilde\rho\,\mathrm dV
=0$.

For every fixed $r>0$, the convergence $\rho_t\to\widetilde\rho$ is
uniform on $T_{V_0}^n\setminus B_r(p)$. Hence
\[
\lim_{t\to1^-}
\int_{T_{V_0}^n\setminus B_r(p)}
\rho_t\,\mathrm dV
=
\int_{T_{V_0}^n\setminus B_r(p)}
\widetilde\rho\,\mathrm dV.
\]
Letting $r\to0$, we obtain
\begin{equation}
\lim_{r\to0}
\int_{T_{V_0}^n\setminus B_r(p)}
\widetilde\rho\,\mathrm dV
=
J_n.
\label{eq:regular-part}
\end{equation}

We next analyze the integral over $B_r(p)$. Let $x=p+y$, then $\gamma_i(p+y)=-\cos(\sqrt n\,y_i)$.
The Taylor expansion at $y=0$ yields
\begin{equation}
1+s(p+y)
=
\frac{1}{2}|y|^2+O(|y|^4),
\label{eq:taylor-s}
\end{equation}
\begin{equation}
1-P(p+y)=O(|y|^4).
\label{eq:taylor-P}
\end{equation}
Consequently, when $|y|<r_0$, there exist $r_0>0$ and constants $L_1,L_2>0$ such
that
\begin{equation}
L_1|y|^2
\leq
1+s(p+y)
\leq
L_2|y|^2,
\label{eq:s-local-bound}
\end{equation}
\begin{equation}
0\leq1-P(p+y)\leq L_2|y|^4.
\label{eq:P-local-bound}
\end{equation}

Set $\sigma=1-t$. By (\ref{eq:s-local-bound}) and
(\ref{eq:P-local-bound}), for $t$ sufficiently close to $1$, there
exist constants $L_3,L_4>0$ such that
\[
\begin{aligned}
1+ts(p+y)
=
\sigma+t\bigl(1+s(p+y)\bigr)
\geq
L_3\bigl(\sigma+|y|^2\bigr),
\end{aligned}
\]
\[
\begin{aligned}
1-t^2P(p+y)=
1-t^2+t^2\bigl(1-P(p+y)\bigr)
\leq
L_4\bigl(\sigma+|y|^4\bigr).
\end{aligned}
\]
For some constant $L_5>0$, it follows that
\begin{equation}
\rho_t(p+y)
\leq
L_5
\frac{\sigma^{n/2}}
     {(\sigma+|y|^2)^n}
+
L_5.
\label{eq:rho-local-bound}
\end{equation}

We now introduce the rescaling $y=\sqrt{2\sigma}\,z$. For every fixed $\widetilde r>0$ and $z\in B_{\widetilde r}(0)$,
equations (\ref{eq:taylor-s}) and (\ref{eq:taylor-P})give
\[
\begin{aligned}
1+ts\bigl(p+\sqrt{2\sigma}\,z\bigr)
=
\sigma
+
(1-\sigma)
\left(
\sigma|z|^2+O(\sigma^2)
\right)
=
\sigma
\left(
1+|z|^2+O(\sigma)
\right),
\end{aligned}
\]
\[
\begin{aligned}
1-t^2P\bigl(p+\sqrt{2\sigma}\,z\bigr)
=
\sigma(2-\sigma)
+
(1-\sigma)^2O(\sigma^2)=
2\sigma\bigl(1+O(\sigma)\bigr).
\end{aligned}
\]
Hence, by \eqref{eq:rho-local-bound} and the dominated convergence theorem, we have
\begin{equation}\label{eq-core-limit}
\lim_{t\to1^-}
\int_{B_{\sqrt{2\sigma}\widetilde r}(p)}
\rho_t\,\mathrm dV
=
\lim_{t\to1^-}
\int_{B_{\widetilde r}(0)}
(2\sigma)^{n/2}
\rho_t\bigl(p+\sqrt{2\sigma}\,z\bigr)
\,\mathrm dz =
\int_{B_{\widetilde r}(0)}
\frac{2^n}{(1+|z|^2)^n}
\,\mathrm dz.
\end{equation}
The limiting density $2^n(1+|z|^2)^{-n}\dd z$ is precisely the volume form of the unit $n$-sphere in stereographic coordinates. Thus, the formula \eqref{eq-core-limit} describes a single spherical bubble forming at $p$, carrying total mean-curvature mass $\omega_n$.

Moreover, for every fixed $\widetilde r>0$, when $t$ is sufficiently close to $1$, one has $B_{\sqrt{2\sigma}\widetilde r}(p)\subset B_r(p)$.
 Since $\rho_t\geq0$,
equation (\ref{eq-core-limit}) implies
\[
\liminf_{t\to1^-}
\int_{B_r(p)}
\rho_t\,\mathrm dV
\geq
\int_{B_{\widetilde r}(0)}
\frac{2^n}{(1+|z|^2)^n}
\,\mathrm dz.
\]
Letting $\widetilde r\to\infty$ and using the stereographic projection
formula, we obtain
\begin{equation}
\liminf_{t\to1^-}
\int_{B_r(p)}
\rho_t\,\mathrm dV
\geq
\int_{\mathbb R^n}
\frac{2^n}{(1+|z|^2)^n}
\,\mathrm dz
=
\omega_n.
\label{eq:local-liminf}
\end{equation}

On the other hand, fix $\widetilde r>1$ and $0<r<r_0$. By (\ref{eq:rho-local-bound}), when $t$ is sufficiently close to $1$, we obtain
\[
\begin{aligned}
\int_{B_r(p)\setminus
B_{\sqrt{2\sigma}\widetilde r}(p)}
\rho_t\,\mathrm dV
\leq
L_5
\int_{|z|\geq\widetilde r}
\frac{2^{n/2}}{(1+|z|^2)^n}
\,\mathrm dz
+
L_5r^n.
\end{aligned}
\]
Combining this estimate with 
(\ref{eq-core-limit}), we find
\[
\begin{aligned}
\limsup_{t\to1^-}
\int_{B_r(p)}
\rho_t\,\mathrm dV
\leq{}
\int_{B_{\widetilde r}(0)}
\frac{2^n}{(1+|z|^2)^n}
\,\mathrm dz+
L_5
\int_{|z|\geq\widetilde r}
\frac{2^{n/2}}{(1+|z|^2)^n}
\,\mathrm dz
+
L_5r^n.
\end{aligned}
\]
Letting $\widetilde r\to\infty$ yields
\begin{equation}
\limsup_{t\to1^-}
\int_{B_r(p)}
\rho_t\,\mathrm dV
\leq
\omega_n+L_5r^n.
\label{eq:local-limsup}
\end{equation}

Combining (\ref{eq:local-liminf}) and (\ref{eq:local-limsup}) with the
convergence away from $p$, we obtain
\[
\begin{aligned}
\int_{T_{V_0}^n\setminus B_r(p)}
\widetilde\rho\,\mathrm dV+\omega_n
\leq
\liminf_{t\to1^-}F_n(t)\leq
\limsup_{t\to1^-}F_n(t)\leq
\int_{T_{V_0}^n\setminus B_r(p)}
\widetilde\rho\,\mathrm dV
+
\omega_n
+
L_5r^n.
\end{aligned}
\]
Finally, letting $r\to0$ and applying (\ref{eq:regular-part}), we conclude
that
\[
\lim_{t\to1^-}F_n(t)
=
J_n+\omega_n.
\]
This proves (\ref{eq:boundary-limit}).
\end{proof}
\begin{remark}\rm
Geometrically, as $\varepsilon$ approaches $1$ (equivalently, $t \to 1$), the mapping 
$\Phi_{\varepsilon v^{\sharp}} \circ f$ undergoes a conformal blow-up, resulting in the bubbling off of an $n$-dimensional sphere. The limit
\begin{equation}
    \lim_{t \to 1^-} F_n(t) = J_n + \omega_n
\end{equation}
indicates that the total mean curvature functional exhibits bubbling and energy quantization along this one-parameter conformal deformation. Here, $J_n$ denotes the total mean curvature of the noncompact Euclidean submanifold obtained by stereographically projecting the Clifford $n$-torus from the pole $-v^\sharp$ onto $\mathbb{R}^{2n-1}$, and $\omega_n$ is the quantized energy contributed by a single standard spherical bubble formed at $p$.
\end{remark}

\begin{remark}\rm
To compare the boundary value with the energy of the original Clifford $n$-torus and with Chen's lower bound $\omega_n$ in \cite{Chen1971TotalI}, define
$$
\mathcal{F}_n
:=
\frac{\displaystyle\lim_{t\to1^-}F_n(t)}
{F_n(0)}
=
\frac{\omega_n+J_n}
{\left(2\pi/\sqrt n\right)^n},\qquad \mathcal{R}_n
:\frac{\displaystyle\lim_{t\to1^-}F_n(t)}
{\omega_n}=\frac{\omega_n+J_n}{\omega_n}.
$$
   We can carry out numerical computations and obtain the following table.
\begin{table}[H]
    \centering
    \begin{tabular}{c|c|c|c|c|c}
        $n$ & $\omega_n$ & $\omega_n+J_n$ & $F_n(0)$ & $\mathcal{F}_n$& $\mathcal{R}_n$
\\ \hline
3 & 19.739 &40.425&47.737& 0.847 &2.048 \\
4 & 26.319 &71.779&97.409& 0.737 &2.727 \\
5 & 31.006 &115.700&175.176& 0.660 & 3.732\\
6  & 33.073&172.327 &284.856&0.605&5.211\\
7  & 32.470&239.399&426.006& 0.562&7.373  \\
8 & 29.687 &312.176&593.033&0.526 & 10.516\\
9  & 25.502&384.049&775.403& 0.495 &15.060 \\
10 & 20.725&447.816 &958.956&0.467&21.608
    \end{tabular}
     \caption{Numerical comparison}
    \label{tab1}
\end{table}


\end{remark}}

\appendix\section{Minimal Submanifolds in $\SB^N$ under M\"{o}bius transformations}
\begin{lemma}\label{eq-minimal}
Let $f: M^n \rightarrow \SB^N$ be a closed minimal submanifold. For any constant unit vector $v \in \mathbb{R}^{N+1}$, let $v^\top$ and $v^\bot$ denote the tangential and normal projections of $v$ to $M$, respectively, and let $B$ denote the second fundamental form of ${M^n}$ in $\SB^N$. Then the following integral identities hold:
\begin{align}
     \int_{M^n} |v^\top|^2 \,\mathrm dV& = n \int_{M^n} \langle f, v \rangle^2 \, \mathrm dV \label{eq-2ord}\\
     \int_{M^n}\langle f, v \rangle |v^\top|^2 \, \mathrm{d}V &= \frac{n}{2} \int_{M^n} \langle f, v \rangle^3 \, \mathrm{d}V \label{eq-3ord}\\
    \int_{M^n} \langle f, v \rangle^2 |v^\top|^2 \, \mathrm{d}V &= \frac{n}{3} \int_{M^n} \langle f, v \rangle^4 \, \mathrm{d}V \label{eq-4ord1}\\ 
    \int_{M^n}|v^\top|^4 \, \mathrm{d}V &= \frac{n(n+2)}{3} \int_{M^n} \langle f, v \rangle^4 \, \mathrm{d}V - 2 \int_{M^n} \langle f, v \rangle \langle B(v^\top, v^\top), v^\bot \rangle \, \mathrm{d}V \label{eq-4ord2}
\end{align}
\end{lemma}
\begin{proof}
Set $u(x) \triangleq \langle f(x), v \rangle$.
Since $v$ is a constant vector, the gradient of $u$ on ${M^n}$ is exactly the tangential projection of $v$, i.e., $\nabla u = v^\top$, which implies $|\nabla u|^2 = |v^\top|^2$.
By Takahashi's theorem, the Laplacian of $u$ takes the form 
\[ 
\Delta u = \langle \Delta f, v \rangle = \langle -n f, v \rangle = -n u. 
\]

For any positive integer $k$, using $\nabla(u^k) = ku^{k-1} \nabla u$ and 
\[ 
\int_{M^n} u^k \Delta u \, \mathrm{d}V = -\int_{M^n} \langle \nabla(u^k), \nabla u \rangle \, \mathrm{d}V,
\]
we obtain 
\[ 
\int_{M^n} u^{k-1} |\nabla u|^2 \, \mathrm{d}V=\frac{n}{k}\int_{M^n} u^{k+1}\, \mathrm{d}V, 
\]
from which \eqref{eq-2ord} $\sim$ \eqref{eq-4ord1} follows. 

Consider the vector field $X = u |\nabla u|^2 \nabla u$. The divergence of $X$ is
\[ 
\operatorname{div}(X) = \langle \nabla(u |\nabla u|^2), \nabla u \rangle + u |\nabla u|^2 \Delta u. 
\]
Expanding the gradient term, we get $\nabla(u |\nabla u|^2) = |\nabla u|^2 \nabla u + u \nabla |\nabla u|^2$. Thus,
\[ 
\operatorname{div}(X) = |\nabla u|^4 + u \langle \nabla |\nabla u|^2, \nabla u \rangle - n u^2 |\nabla u|^2. 
\]
Integrating over ${M^n}$ and applying the divergence theorem $\int_{M^n} \operatorname{div}(X) \, \mathrm{d}V = 0$, we obtain
\[ 
\int_{M^n} |\nabla u|^4 \, \mathrm{d}V = n \int_{M^n} u^2 |\nabla u|^2 \, \mathrm{d}V - \int_{M^n} u \langle \nabla |\nabla u|^2, \nabla u \rangle \, \mathrm{d}V. 
\]

To compute the term $\langle \nabla |\nabla u|^2, \nabla u \rangle$, let $Y$ and $Z$ be tangent vector fields on ${M^n}$, recall that $\frac{1}{2} \langle \nabla |\nabla u|^2, Y \rangle = \operatorname{Hess}(u)(\nabla u, Y)$. For the submanifold ${M^n} \subset \mathbb{S}^N \subset \mathbb{R}^{N+1}$, the Gauss formula yields the Hessian of the restriction of a linear function, 
\[ 
\operatorname{Hess}(u)(Y, Z) = -u \langle Y, Z \rangle + \langle B(Y, Z), v \rangle=-u \langle Y, Z \rangle + \langle B(Y, Z), v^\perp \rangle, 
\]
where $B$ is the second fundamental form of ${M^n}$ in $\SB^N$. 
Therefore,
\[ 
\langle \nabla |\nabla u|^2, \nabla u \rangle = 2 \operatorname{Hess}(u)(\nabla u, \nabla u) = -2u |\nabla u|^2 + 2\langle B(\nabla u, \nabla u), v^\bot \rangle. 
\]
Substituting this back into the integral equation gives
\begin{align*}
\int_{M^n} |\nabla u|^4 \, \mathrm{d}V &= n \int_{M^n} u^2 |\nabla u|^2 \, \mathrm{d}V - \int_{M^n} u \left( -2u |\nabla u|^2 + 2\langle B(\nabla u, \nabla u), v^\bot \rangle \right) dV \\
&= (n+2) \int_{M^n} u^2 |\nabla u|^2 \, \mathrm{d}V - 2 \int_{M^n} u \langle B(\nabla u, \nabla u), v^\bot \rangle \, \mathrm{d}V.
\end{align*}
Using the result from (\ref{eq-4ord1}) that $\int_{M^n} u^2 |\nabla u|^2 \, \mathrm{d}V = \frac{n}{3} \int_{M^n} u^4 \, \mathrm{d}V$, we conclude
\[ 
\int_{M^n} |v^\top|^4 \, \mathrm{d}V = \frac{n(n+2)}{3} \int_{M^n} \langle f, v \rangle^4 \, \mathrm{d}V - 2 \int_{M^n} \langle f, v \rangle \langle B(v^\top, v^\top), v^\bot \rangle \, \mathrm{d}V. 
\]
This completes the proof.
\end{proof}

$\mathbf{Declaration~ on~ the~ use~of~ AI}$: 
During the revision of this manuscript, ChatGPT 5.6 Sol was used to generate preliminary versions of certain arguments in Lemma 3.2, Remark 4.6, and Proposition 5.5, to identify possible gaps, and to assist with language editing. All AI-generated material was independently checked, substantially revised, and rewritten by the authors, who take full responsibility for all mathematical content.\\

$\mathbf{Acknowledgement}$: The second named author is  supported by the Natural Science Foundation of Fujian Province of China (Grant No. 2026J002028) and  NSFC (Grant No. 12371052). The third named author is supported by NSFC (Grant No. 12171473) and the Fundamental Research Funds for Central Universities. \hfill



\def\refname

\end{document}